\documentclass[twoside,letter,final]{siamltex}
\usepackage{amsmath}
\usepackage{mathtools}
\usepackage{amssymb}
\usepackage{amsfonts}
\usepackage{amsxtra}
\usepackage{amstext}
\usepackage{amsbsy}
\usepackage{amscd}
\usepackage{graphicx}
\usepackage{float}
\usepackage{cite}
\usepackage{xcolor}
\usepackage{srcltx} %Inverse search
\usepackage{marginnote,slashbox}
\usepackage{rotating}
\usepackage{comment}
\definecolor{darkblue}{rgb}{0.0,0.0,0.6}
\definecolor{darkgreen}{rgb}{0.0,0.6,0.0}
\usepackage[pdftex,colorlinks=true,urlcolor=darkblue,citecolor=darkblue,linkcolor=darkblue]{hyperref}
\usepackage[utf8]{inputenc}      % for German umlauts
\usepackage{booktabs}            % to get fancier tables
\usepackage{url}                 % to typeset URLs
\usepackage[T1]{fontenc}         % to enable correct hyphenation of words with umlauts
\usepackage{algorithmic}         % defines algorithmic environment
\usepackage{pgfplots}
\pgfplotsset{width=6.5cm,compat=newest}
\usepackage{pgfplotstable}

\usepackage{enumitem}

\usepackage{calc}
\usepackage[normalem]{ulem}

\newtheorem{algorithm}[theorem]{Algorithm}

\numberwithin{table}{section}    % for Table 1.1
\numberwithin{figure}{section}   % for Figure 1.1
\numberwithin{equation}{section} % for equation (1.1)

\usepackage{my_latex_commands}

\renewcommand{\clos}[1]{\operatorname{cl}\left(#1\right)}

\newtheorem{assumption}[theorem]{Assumption}
\newtheorem{remark}[theorem]{Remark}
\newenvironment{proofof}[1]{\emph{Proof of #1.}}{\endproof}

\begin{document}
\title{A non-smooth trust-region method for locally Lipschitz functions with application to optimization problems constrained by variational inequalities}
\date{\today}
\author{C.~Christof\footnotemark[3] \and J.~C.~De Los Reyes\footnotemark[4] \and C.~Meyer\footnotemark[3]}
\renewcommand{\thefootnote}{\fnsymbol{footnote}}
\footnotetext[3]{Faculty of Mathematics, Technische Universit\"at Dortmund, Germany.}
\renewcommand{\thefootnote}{\fnsymbol{footnote}}
\footnotetext[4]{Research Center on Mathematical Modelling (MODEMAT), Escuela Politécnica Nacional, Quito, Ecuador}

\maketitle
\begin{abstract}
  We propose a non-smooth trust–region method for solving optimization problems with locally Lipschitz continuous functions, with application to problems constrained by variational inequalities of the second kind. Under suitable assumptions on the model functions, convergence of the general algorithm to a C– stationary point is verified. For variational inequality constrained problems, we are able to properly characterize the Bouligand subdifferential of the reduced cost function and, based on that, we propose a computable trust–region model which fulfills the convergence hypotheses of the general algorithm. The article concludes with the experimental study of the main properties of the proposed method based on two different numerical instances.
\end{abstract}


%%%%%%%%%%%%%%%%%%%%%%%%%%%%%%%%%%%%%%%%%%%%%%%%%%%%%%%%%%%
\section{Introduction}

The study of optimization problems with variational inequality (VI) constraints is a challenging topic in mathematical programming, due to the intricate structure of the type of stationary points one aims to reach. Whereas for differentiable problems only one type of stationarity takes place, in the nonsmooth framework a family of concepts arise (e.g., Clarke, Dini, Bouligand or Mordukhovich stationarity), each one with advantages and shortcomings.

To overcome the difficulties related to the nonsmoothness of these types of problems, relaxation or penalization approaches have been frequently proposed, with different outcomes concerning optimality conditions and solution algorithms. The main criticism to these approaches, however, is that by removing/relaxing the nonsmoothness, the structure of the original problem is altered, possibly leading to undesired or unphysical solutions.

In the special case of \emph{Mathematical Programs with Equilibrium Constraints (MPEC)} intensive efforts have been carried out and different methods proposed or extended (see \cite{luo} and the references therein). Typically, convergence to C-stationary points can be guaranteed for the proposed algorithms. The search for stronger (like B-- or M--) stationary points remains a challenge.

In the case of optimization problems with variational inequality constraints of the second kind, much less work has been carried out. In \cite{delosreyes2011} a semismooth Newton method based on a regularized version of the problem was proposed and superlinear convergence verified. The solution, however, corresponds to a regularized problem and not to the original one, although consistency is also proved there. A first attempt to find solutions without using regularization was tested in \cite{delosreyesmeyer}, where a trust--region algorithm was considered with promising results.

Trust-region methods have been investigated for nonsmooth optimization of locally Lipschitz continuous functions in, e.g., \cite{sunyuan,dennis,outratazowe,qisun,noll}. The underlying hypotheses, however, are difficult to verify for optimization problems with variational inequality constraints. This especially concerns the hypothesis that the cost function has to be \emph{regular} (see \cite{scholtesstohr} for a discussion on this matter), which happens to be very restrictive for problems as the ones considered in this manuscript. In combination with bundle methods, a convergence theory based on weaker assumptions, similar to ours, was recently studied in \cite{noll}.

% Apart of the convergence analysis, the choice of an easy to compute model candidate still appears to be a challenge in each case. Theoretical results have been proved for Clarke's generalized directional derivative. \textcolor{red}{Referenzen hierf\"ur? Ist das korrekt? Eigentlich entspricht das Modell in \cite[Section4.9]{qisun} der Clarkeschen Richtungsableitung und das ist ja gerade kein gutes Modell, wie unser 1D-Gegenbeispiel zeigt. Ich wuerde den Absatz vielleicht einfach weglassen} This candidate, however, turns out to be difficult to use in most of the relevant practical cases, in particular in nonlinear optimization with variational inequality constraints.

The purpose of this paper is twofold. First, we propose a general non-smooth trust-region method for locally Lipschitz continuous functions, and carry out the convergence analysis of the approach. Similarly to \cite{ayp, noll}, we prove convergence to C-stationary points without assuming regularity of the cost function. An essential and novel feature of our algorithm is the computation of the quality indicator, based on a comparison between an ``easy'' local model and a complicated one containing neighborhood information (see \ref{it:quality} below).

Second, and maybe more important, we consider the application of the proposed algorithm to optimal control problems governed by variational inequalities of the second kind. To that end, the Bouligand subdifferential of the control-to-state map
is precisely characterized, which allows to construct a model function which satisfies the
hypotheses of our general trust-region method. Differently from other contributions were the choice of a good candidate is assumed, we provide, for this special family of problems, a way to find them explicitely.

The paper is organized as follows. In Section~\ref{sec:TRgeneral} the general algorithm is presented, together with the assumptions on the model function and a general convergence result.
In Section~\ref{sec:composite}, we show how to construct a suitable model function for the case
of composite functions with a non-smooth inner function as they appear in the implicit programming approach for optimal control of VIs. Afterwards, in Section~\ref{sec:optvi},
we apply these findings to an optimal control problem governed by a VI of the second kind.
The construction of the model function associated with this example is based on a precise
characterization of the Bouligand-subdifferential of the control-to-state mapping.
In Section~\ref{section: experiments}, numerical experiments are carried out to verify the main properties of the proposed algorithm. The paper ends with some concluding remarks.

\subsection{Notation}
By $\lambda^n$, we denote the $n$-dimensional Lebesgue-measure.
Moreover, $\|\cdot \|$ and $\dual{\cdot}{\cdot}$ denote the Euclidean norm and the associated
scalar product, whereas $\| \cdot \|_\infty$ and $\| \cdot \|_1$ stand for
the maximum and 1-norm, respectively. In addition, $\|\cdot\|_{\R^{n\times n}}$
denotes the spectral norm, and we sometimes suppress the index $\R^{n\times n}$, if no
ambiguity is possible. Given $x\in \R^n$ and $r>0$, we denote by $B_r(x)$ the \emph{closed} ball around $x$ with radius $r$.

%%%%%%%%%%%%%%%%%%%%%%%%%%%%%%%%%%%%%%%%%%%%%%%%%%%%%%
\section{A Non-Smooth Trust-Region Algorithm}\label{sec:TRgeneral}

We consider the general nonlinear optimization problem
\begin{equation}\tag{P}\label{eq:p}
 \min_{x\in \R^n} \; f(x),
\end{equation}
with an objective function satisfying the following conditions:

\begin{assumption}[Objective function]\label{assu:f}
 The function $f: \R^n \to \R$ is supposed to be \emph{locally Lipschitz continuous},
 i.e., for all $x\in \R^n$ there exist $\delta > 0$ and $L > 0$ so that
 \begin{equation*}
  |f(y) - f(z)| \leq L\, \|y-z\| \quad \forall\,y,z\in B_\delta(x).
 \end{equation*}
\end{assumption}

Next, we introduce some basic concepts of nonsmooth optimization that are going to be used along the paper.

\begin{definition}[Subdifferentials]
 Let $F:\R^n \to \R^m$, $m, n\in \N$, be locally Lipschitz-continuous.
 We denote the set of points, where $F$ is differentiable, by $\DD_F$.
 By Rademacher's theorem, this set is dense in $\R^n$. For a given $x\in \R^n$ we define
 \begin{itemize}
  \item the \emph{Bouligand-subdifferential} by
  \begin{equation} \label{def:subdiff}
   \partial_B F(x) := \{ G \in \R^{m\times n} :
   \exists\, (x_n)\subset \DD_F \text{ with } x_n \to x,\; F'(x_n) \to G \},
  \end{equation}
  \item the \emph{Clarke-subdifferential} by
  \begin{equation}
   \partial F(x) = \clos{\conv(\partial_B F(x))},
  \end{equation}
  where $\conv$ denotes the convex hull.
 \end{itemize}
\end{definition}

It is well known that in case of a scalar-valued locally Lipschitz-continuous function $f:\R^n\to \R$
the Clarke-subdifferential can equivalently be expressed as
\begin{equation*}
 \partial f(x) = \{ g\in \R^n: \dual{g}{v} \leq f^\circ(x;v) \quad\forall\, v\in \R^n \},
\end{equation*}
where $f^\circ$ denotes Clarke's generalized directional derivative. For scalar-valued functions
we moreover define the following notion of stationarity:

\begin{definition}\label{def:stat}
 Let $f:\R^n \to \R$, $n\in \N$, be locally Lipschitz-continuous. We then call a point $\bar x\in \R^n$
 C(larke)-stationary, if $0\in \partial f(\bar x)$.
\end{definition}

Our algorithm is based on a suitably chosen model function, whose existence is assumed as a start.
In the later sections, we will see how to construct such a model for concrete problems.
To be more precise, we require the following.

\begin{assumption}[Model function]\label{assu:model}\
 \begin{enumerate}
  \item For every $x \in \R^n$, we can calculate a subgradient $g \in \partial f(x) $.
  \item\label{it:model}
  There is a model function $\phi: \R^n \times \R^+ \times \R^n \to \R$ satisfying the following conditions:
  \begin{enumerate}
   \item\label{it:semi} For every $(x,\Delta) \in \R^n \times \R^+$, the mapping
   $\R^n \ni d \mapsto \phi(x, \Delta; \cdot\,)$ is positively homogeneous and lower semicontinuous.
   \item\label{it:cstat} \emph{Stationarity indicator property:}\\
  The stationarity measure defined through
   \begin{equation}\label{eq:normersatz}
    \psi(x, \Delta)  := - \min_{\|d\|\leq 1} \phi(x, \Delta; d) \geq 0
   \end{equation}
   satisfies the following:\\
   If a sequence $\{x_k, \Delta_k\} \subset \R^n \times \R^+$ satisfies
   \begin{equation*}
    x_k \to x, \quad \Delta_k \to 0, \quad \text{and} \quad \psi(x_k, \Delta_k) \to 0,
   \end{equation*}
   then it follows that $0 \in \partial f(x)$.
   \item\label{it:remainder} \emph{Remainder term property:}\\
   For every sequence  $\{x_k, \Delta_k\} \subset \R^n \times \R^+$ satisfying
   \begin{equation*}
    x_k \to x, \quad \Delta_k \to 0, \quad \text{and} \quad \lim_{k\to\infty} \psi(x_k, \Delta_k) > 0,
   \end{equation*}
   there holds
   \begin{equation}\label{eq:remainder}
   \limsup_{k\to\infty} \sup_{d\in B_{\Delta_k}(0)}
   \frac{f(x_k + d) - f(x_k) - \phi(x_k, \Delta_k; d)}{\Delta_k} \leq 0.
   \end{equation}
  \end{enumerate}
 \end{enumerate}
\end{assumption}

Note that the inequality in \eqref{eq:normersatz} follows immediately from the positive homogeneity and the lower semicontinuity of
$\phi(x, \Delta, \cdot\,)$.
Note moreover that the limes superior in \eqref{eq:remainder} is actually a limes, since
the inner supremum is always greater or equal zero (just choose $d = 0 \in B_{\Delta_k}(0)$).

\begin{remark}
 Assumption~\ref{assu:model}(\ref{it:model}) is closely related to the hypotheses required in other
 contributions in the field of non-smooth trust-region methods such as in \cite[Assumption~A2]{qisun} or \cite[Theorem~1]{noll}.
 In particular, condition~(\ref{it:remainder}) is similar to the assumption that the objective $f$ admits a
 \emph{strict first-order model}, cf.~\cite[Definition~1 and Axiom~$(\widetilde{M_2})$]{noll}. To be more precise,
 it is easy to see that \cite[$(\widetilde{M_2})$]{noll} is sufficient for the remainder term property in (\ref{it:remainder}).
 This property reflects the fact that the model function has to incorporate certain information
 about the objective function in a neighborhood of the current iterate.
\end{remark}

Given the model function, our algorithm reads as follows:

\begin{algorithm}[Non-Smooth Trust-Region Algorithm]\label{alg:TR}
 \begin{algorithmic}[1]
  \STATE\label{it:init} Initialization:\\
  Choose constants
  \begin{equation*}
   \Delta_{\min} > 0, \quad 0 < \eta_1 < \eta_2 < 1, \quad 0 < \beta_1 < 1 < \beta_2,
   \quad 0 < \mu \leq 1,
  \end{equation*}
  an initial value $x_0 \in \R^n$, and an initial TR-radius $\Delta_0 > \Delta_{\min}$. Set $k=0$.
  \FOR{$k=0, 1, 2, ...$}
   \STATE\label{it:Hk} Choose a subgradient $g_k \in \partial f(x_k)$ and a matrix $H_k \in \R^{n\times n}_{\textup{sym}}$.
   \IF{$g_k = 0$}\label{it:gknull}
    \STATE STOP the iteration, $x_k$ is C-stationary, i.e., $0 \in \partial f(x_k)$.
   \ELSE
    \IF{$\Delta_k \geq \Delta_{\min}$} \label{it:deltamin}
     \STATE Compute an inexact solution $d_k$ of the \emph{trust-region subproblem}
     \begin{equation}\tag{Q$_{k}$}\label{eq:qk}
      \left.
      \begin{aligned}
       \min_{d\in \R^n} & \quad q_k(d) := f(x_k) + \dual{g_k}{d} + \frac{1}{2}\, d^\top H_k d\\
       \text{s.t.} & \quad \|d\| \leq \Delta_k,
      \end{aligned}
      \qquad\right\}
     \end{equation}
     that satisfies the \emph{generalized Cauchy-decrease condition}
     \begin{equation}\label{eq:cauchy}
      f(x_k) - q_k(d_k) \geq
      \frac{\mu}{2}\,\|g_k\|\,\min\Big\{ \Delta_k, \frac{\|g_k\|}{\|H_k\|} \Big\}.
     \end{equation}
     \STATE Compute the quality indicator
     \begin{equation*}
      \rho_k := \frac{f(x_k) - f(x_k + d_k)}{f(x_k) - q_k(d_k)}.
     \end{equation*}
    \ELSE
     \STATE\label{step:TRmod} Compute an inexact solution $d_k$ of the following \emph{modified trust-region subproblem}
     \begin{equation}\tag{$\tilde{\textup{Q}}_k$}\label{eq:tildeqk}
      \left.
      \begin{aligned}
       \min_{d\in \R^n} & \quad \tilde q_k(d):= f(x_k) + \phi(x_k, \Delta_k; d) + \frac{1}{2}\, d^\top H_k d\\
       \text{s.t.} & \quad \|d\| \leq \Delta_k
      \end{aligned}
      \qquad\right\}
     \end{equation}
     that satisfies the \emph{modified Cauchy-decrease condition}
     \begin{equation}\label{eq:tildecauchy}
      f(x_k) - \tilde q_k(d_k) \geq
      \frac{\mu}{2}\,\psi(x_k, \Delta_k) \,\min\Big\{ \Delta_k, \frac{\psi(x_k, \Delta_k)}{\|H_k\|} \Big\},
     \end{equation}
     where $\psi(x_k, \Delta_k)$ is as defined in \eqref{eq:normersatz}.
     \STATE\label{it:quality} Compute the modified quality indicator
     \begin{equation}\label{eq:rhomod}
      \rho_k :=
      \begin{cases}
       \displaystyle{\frac{f(x_k) - f(x_k + d_k)}{f(x_k) - \tilde{q}_k(d_k)}},
       & \text{if }  \psi(x_k, \Delta_k) > \|g_k\|\, \Delta_k \\
       0, & \text{if }  \psi(x_k, \Delta_k) \leq \|g_k\|\, \Delta_k.
      \end{cases}
     \end{equation}
    \ENDIF
    \STATE Update: Set
    \begin{align}
     x_{k+1} & :=
     \begin{cases}
      x_k, & \text{if } \rho_k \leq \eta_1 \quad \text{(null step)},\\
      x_k + d_k, & \text{otherwise} \quad \text{(successful step)},
     \end{cases} \label{eq:upx}\\
     \Delta_{k+1} &:=
     \begin{cases}
      \beta_1\,\Delta_k, & \text{if } \rho_k \leq \eta_1,\\
      \max\{\Delta_{\min}, \Delta_k\}, & \text{if } \eta_1 < \rho_k \leq \eta_2,\\
      \max\{\Delta_{\min}, \beta_2 \Delta_k\}, & \text{if } \rho_k > \eta_2.
     \end{cases} \label{eq:updelta}
    \end{align}
    Set $k = k+1$.
   \ENDIF
  \ENDFOR
 \end{algorithmic}
\end{algorithm}

\begin{remark}[Bouligand-subgradients]
 Our convergence analysis in Section~\ref{subsec:conv} below does not require to choose a particular
 subgradient in Step~\ref{it:Hk}, it works for every element of the Clarke-subdifferential.
 Therefore, one could well restrict to elements of the Bouligand-subdifferential in this step.
 In this case, the termination criterion in Step~\ref{it:gknull} would imply that
 $0 \in \partial_B f(x_k)$, i.e., a stationarity condition which is in general only meaningful
 in smooth points.\\
 In the applications we are interested in, we exactly proceed in this way and choose
 elements of the Bouligand-subdifferential in Step~\ref{it:Hk},
 cf.~Algorithm \ref{alg:trvi} below.
\end{remark}

\begin{remark}[Comparison to other non-smooth trust region algorithms]
 The essential differences to other non-smooth trust-region algorithms such as for instance the ones
 presented in \cite{dennis, qisun, sunyuan, noll} are the following:
 \begin{itemize}
  \item By introducing the distinction of cases in step~\ref{it:deltamin},
  we allow for a classical trust-region subproblem, which is easy to solve by standard methods such as
  the dogleg method or Steinhaug's CG method. As our numerical experiments in Section~\ref{section: experiments} show,
  in the applications we have in mind, the second case involving the complicated model
  in \eqref{eq:tildeqk} occurs only if $\Delta_{min}$ is chosen comparatively large. Our algorithm therefore accounts for
  the fact that many non-smooth problems can well be solved by classical trust-region methods
  and only switches to complicated model functions if it is strictly necessary.
  \item Another essential feature of the algorithm, which ensures the convergence of the method,
  is the computation of the quality indicator in step~\ref{it:quality}. It basically
  corresponds to a comparison of the ``easy'' and the complicated model weighted with the
  trust-region radius. Since the complicated model contains neighborhood information of the
  objective function, it may become insufficiently accurate to measure stationarity.
  This issue is resolved by the comparison with the local ``easy'' model in step~\ref{it:quality}.
 \end{itemize}
\end{remark}

Note that the modified trust-region subproblem \eqref{eq:tildeqk} admits an optimal solution
due to the lower semicontinuity of $\phi$ w.r.t.\ the last argument by Assumption~\ref{assu:model}(\ref{it:semi}).
Moreover, it is always possible to find inexact solutions to \eqref{eq:qk} and \eqref{eq:tildeqk}
that fulfill the respective Cauchy-decrease conditions.

\begin{lemma}\label{lem:cauchyglob}
 Global minimizers of \eqref{eq:qk} and \eqref{eq:tildeqk} satisfy the respective
 Cauchy-decrease conditions in \eqref{eq:cauchy} and \eqref{eq:tildecauchy} for every $\mu \leq 1$.
\end{lemma}

\begin{proof}
 Since our model function $\phi$ is assumed to be positively homogeneous, we can argue as in
 \cite[Lemma~3.2]{qisun}, which immediately gives the assertion.
\hfill\end{proof}

%%%%%%%%%%%%%%%%%%%%%%%%%%%%%%%%%%%%%%%%%%%%%%%%%%%%%%%%%%%%
\subsection{Convergence Analysis}\label{subsec:conv}

In what follows, we show that accumulation points of the sequence of iterates are C-stationary
as defined in Definition~\ref{def:stat}. For this purpose, we need the following

\begin{assumption}[Hessian approximation]\label{assu:hesse}
 The matrices $H_k\in \R^{d\times d}_{\textup{sym}}$ from Step~\ref{it:Hk} of Algorithm \ref{alg:TR} are supposed to satisfy
 \begin{equation*}
  \|H_k\| \leq C_H \quad \forall k \in \N
 \end{equation*}
 with a constant $C_H > 0$.
\end{assumption}

\begin{proposition}\label{prop:proposition42}
 Assume that Algorithm \ref{alg:TR} does not terminate in finitely many iterations.
 Let $(x_k)$ be the sequence of iterates generated by Algorithm \ref{alg:TR}, and suppose that $(x_{k_l})$ is a subsequence of $(x_k)$ satisfying
 \begin{equation*}
  x_{k_l} \to \bar x \quad \text{and} \quad \Delta_{k_l} \to 0 \quad \text{as } l \to \infty.
 \end{equation*}
 Then $0\in \partial f(\bar x)$ holds true.
\end{proposition}

\begin{proof}
 Since $\Delta_{k_l} \to 0$, there exists an $L \in \mathbb{N}$ such that $\Delta_{k_l} < \beta_1 \Delta_{\min}$ for all $l \geq L$.
 This is only  possible if the iterations $k_l - 1$, $l \geq L$, are all null steps, i.e., for all $l \geq L$, we have
 \begin{equation}\label{eq:kminus1}
  x_{k_l - 1} = x_{k_l},\qquad \Delta_{k_l } = \beta_1 \Delta_{k_l - 1} < \beta_1 \Delta_{min}, \qquad \rho_{k_l - 1 } < \eta_1 < 1.
 \end{equation}
 This shows
 \begin{equation*}
  x_{k_l-1} \to \bar x \quad \text{and}\quad \Delta_{k_l-1} \to 0.
 \end{equation*}
 We next show $\psi(x_{k_l-1}, \Delta_{k_l-1}) \to 0$. Once this is established,
 the assertion immediately follows from Assumption~\ref{assu:model}\eqref{it:cstat}.
 For this end, we argue by contradiction and assume that there is an $\varepsilon > 0$ so that
 \begin{equation}\label{eq:limsuppsi}
  \limsup_{l\to\infty} \psi(x_{k_l-1}, \Delta_{k_l-1}) \geq \varepsilon.
 \end{equation}
 Consider now the subsequence of $(x_{k_l-1}, \Delta_{k_l-1})$, which attains the limes superior,
 denoted for simplicity by $(x_m, \Delta_m)_{m\in M}$.
 Then, for $m\in M$ sufficiently large, we have $\psi(x_m, \Delta_m) \geq \varepsilon/2$.
 Since, in addition, the local Lipschitz-continuity of $f$ and $x_m \to \bar x$
 for $m\in M \to \infty$ imply that $\|g_m\|\leq L(\bar x)$, where $L(\bar x)$ denotes the local Lipschitz constant,
 the convergence of $\Delta_m$ to 0 implies that $\psi(x_m, \Delta_m) > \|g_m\|\, \Delta_m$
 for $m\in M$ sufficiently large. Therefore, the first case in \eqref{eq:rhomod} applies in the
 computation of the quality indicator. Moreover, the modified Cauchy-decrease condition in
 \eqref{eq:tildecauchy} and \eqref{eq:normersatz} imply $f(x_m) - \tilde q_m(d_m) > 0$.
 Thus, for all $m \in \N$ sufficiently large, we obtain
 \begin{equation*}
 \begin{aligned}
  \rho_m %&= 1 - \frac{f(x_m + d_m) - \tilde q_m(d_m)}{f(x_m) - q_m(d_m)} \\
  & =  1 - \frac{f(x_m + d_m) - f(x_m) - \phi(x_m, \Delta_m; d_m) - \frac{1}{2}d_m^\top H_m d_m}{f(x_m) - \tilde q_m(d_m)} \\
  & \geq 1 - \frac{\sup_{d\in B_{\Delta_m}(0)} \big(f(x_m + d) - f(x_m) - \phi(x_m, \Delta_m; d)\big) + \frac{1}{2}\, C_H\,\Delta_m^2}
{f(x_m) - \tilde q_m(d_m)} \\
  & \geq 1 - \frac{\sup_{d\in B_{\Delta_m}(0)} \big(f(x_m + d) - f(x_m) - \phi(x_m, \Delta_m; d)\big) + \frac{1}{2}\, C_H\,\Delta_m^2}
  {\frac{\mu}{4}\,\varepsilon\,\min\Big\{ \Delta_m,\frac{\varepsilon}{2 C_H} \Big\}},
 \end{aligned}
 \end{equation*}
 where we used \eqref{eq:tildecauchy} and \eqref{eq:limsuppsi} for the last estimate.
 Note that the supremum in the enumerator is always greater or equal zero, since $d=0$ is feasible.
 Assumption~\ref{assu:model}\eqref{it:remainder} then implies
 \begin{equation*}
 \begin{aligned}
  \liminf_{m\in M \to \infty} \rho_m
  \geq 1 & - \frac{2}{\mu\,\varepsilon}\, C_H\,\lim_{m\in M\to \infty} \Delta_m \\
  & - \frac{4}{\mu\,\varepsilon}\, \limsup_{m\in M \to \infty} \,
  \sup_{d\in B_{\Delta_m}(0)} \frac{f(x_m + d) - f(x_m) - \phi(x_m, \Delta_m; d)}{\Delta_m} \geq 1,
 \end{aligned}
 \end{equation*}
 which however contradicts the last inequality in \eqref{eq:kminus1}.
 Therefore, \eqref{eq:limsuppsi} is not true, which, together with the non-negativity of
 $\psi$ results in
 \begin{equation}\label{eq:psiconv}
  0 \leq \liminf_{l\to\infty} \psi(x_{k_l-1}, \Delta_{k_l-1})
  \leq \limsup_{l\to\infty} \psi(x_{k_l-1}, \Delta_{k_l-1}) = 0.
 \end{equation}
 This finally yields the desired convergence of $\psi(x_{k_l-1}, \Delta_{k_l-1})$,
 which establishes the assertion.
\hfill\end{proof}

\begin{lemma}\label{lem:fconv}
 Assume that Algorithm \ref{alg:TR} does not terminate in finitely many steps.
 If the sequence of iterates $(x_k)$ admits an accumulation point, then the sequence of function values $(f(x_k))$ converges to some $\bar f \in \R$.
\end{lemma}

\begin{proof}
 The arguments are classical. By construction, the sequence $(f(x_k))$ is monotonically decreasing
 so that $f(x_k) \to \bar f \in \R\cup \{-\infty\}$. If a subsequence $(x_{k_l})$ converges to
 a point $\bar x\in \R^n$, then the continuity of $f$ implies $\bar f = f(\bar x) > -\infty$,
 which yields the claim.
\hfill\end{proof}

\begin{theorem}\label{thm:TRconv}
 Assume that Algorithm \ref{alg:TR} does not terminate in finitely many steps.
 Then every accumulation point of the sequence of iterates is C-stationary.
\end{theorem}

\begin{proof}
 If the number of successful iterations is finite, then there is an $N\in \N$ so that
 all iterations $k\geq N$ are null steps. According to the update rule for null steps,
 it follows that $x_k \to x_N =: \bar x$ and $\Delta_k \to 0$ and thus,
 Proposition~\ref{prop:proposition42} yields that $\bar x$ is C-stationary.

 We can thus focus on the case with infinitely many successful iterations.
 Let $\bar x$ be an arbitrary accumulation point of the sequence of iterates and
 denote the corresponding convergent subsequence by $(x_{k_l})$.
 W.l.o.g.\ we may suppose that the iterations $k_l$ are all successful
 (else, we just shift the index forth to the next successful iteration, which does not change the sequence
 due to the update rule for null steps). Since the iterations are successful,
 the monotonicity of $(f(x_k))$ and the
 Cauchy-decrease conditions in \eqref{eq:cauchy} and \eqref{eq:tildecauchy} imply
 \begin{align*}
  f(x_{k_{l}} ) - f(x_{k_{l+1} } ) &\geq f(x_{k_{l} } ) - f(x_{k_{l} + 1} ) \\
  &\geq \eta_1 \,\frac{\mu}{2}\, \nu(x_{k_l} , \Delta_{k_l } ) \min\left\{\Delta_{k_l},
  \frac{\nu(x_{k_l} , \Delta_{k_l} ) }{C_H}  \right\} \geq 0
 \end{align*}
 with
 \begin{equation}\label{eq:nudef}
  \nu(x_{k_l} , \Delta_{k_l } ) :=
  \begin{cases}
   \|g_{k_l}\|, & \text{if } \Delta_{k_l} \geq \Delta_{\min} \\
   \psi (x_{k_l }, \Delta_{k_l}),  & \text{if } \Delta_{k_l} < \Delta_{\min}.
  \end{cases}
 \end{equation}
 Since the sequence $(f(x_k))$ converges by Lemma~\ref{lem:fconv}, it follows
 \begin{equation*}
  \lim_{l \to \infty} \left (
  \nu(x_{k_l} , \Delta_{k_l } )
  \min \left\{\Delta_{k_l }, \frac{\nu(x_{k_l} , \Delta_{k_l} ) }{C_H}  \right\} \right ) = 0,
 \end{equation*}
 i.e., it has to hold
 \begin{equation}\label{eq:nuconv}
  \min \left\{ \Delta_{k_l }    ,  \nu(x_{k_l} , \Delta_{k_l} )   \right\} \to 0
 \end{equation}
 as $l \to \infty$. We now distinguish between three cases:\\
 (i) If there exists a subsequence of $(x_{k_l})$ (unrelabeled for simplicity)
 such that the associated $\Delta_{k_l}$ converge to zero,
 then the claim follows immediately from Proposition~\ref{prop:proposition42}.\\
 (ii) If there exists a subsequence of $(x_{k_l})$ (again unrelabeled)
 such that $\Delta_{k_l} \geq \Delta_{\min}$, then \eqref{eq:nudef} and \eqref{eq:nuconv} imply
 $\|g_{k_l}\| \to 0$. In view of \cite[Prop.~2.1.5(b)]{clarke}, this implies $0\in \partial f(\bar x)$ as claimed.\\
 (iii)  If there exists a subsequence of $(x_{k_l})$ (again unrelabeled)
 with $\varepsilon \leq \Delta_{k_l} < \Delta_{\min}$  for some $\varepsilon > 0$,
 then \eqref{eq:nuconv} gives $\nu(x_{k_l }, \Delta_{k_l}) \to 0$.
 We know, however, that the steps  $(x_{k_l})$ are all successful and,
 according to \eqref{eq:rhomod}, this is only the case if
 \begin{equation*}
  \nu (x_{k_l}, \Delta_{k_l}) = \psi(x_{k_l}, \Delta_{k_l})
  \geq \|g_{k_l}\| \Delta_{k_l} \geq \|g_{k_l}\|\, \varepsilon \geq 0.
 \end{equation*}
 Accordingly, $\|g_{k_l}\| \to 0$ holds and we can argue as in the second case to obtain the claim.
\hfill\end{proof}

\begin{remark}
 The proofs of Proposition~\ref{prop:proposition42} and Theorem~\ref{thm:TRconv},
 in particular \eqref{eq:psiconv} and the distinction of cases after \eqref{eq:nuconv},
 do not only show that every accumulation point is C-stationary, but also that, for every convergent subsequence
 $(x_{k_l})$, the \emph{stationarity indicator} $\min\{\|g_{k_l}\|, \psi(x_{k_l}, \Delta_{k_l})\}$
 converges to zero, which is important for practical reasons, as it lays the foundation for
 an implementable termination criterion of the form
 \begin{equation*}
  \min\{\|g_{k_l}\|, \psi(x_{k_l}, \Delta_{k_l}) \} \leq \texttt{TOL}
 \end{equation*}
 with a given tolerance $\texttt{TOL} > 0$.
\end{remark}

%%%%%%%%%%%%%%%%%%%%%%%%%%%%%%%%%%%%%%%%%%%%%%%%%%%%%

\subsection{A Pathological One-Dimensional Example}

A crucial question in the context of Algorithm~\ref{alg:TR} of course concerns the choice of the model function.
For a general non-smooth problem of the form \eqref{eq:p}, a naive choice would be
\begin{equation}\label{eq:wrongmod}
 \widetilde{\phi}(x, \Delta; d) := \max_{g\in \partial f(x)} \dual{g}{d},
\end{equation}
which is the model proposed in \cite[Section~4.1]{qisun}.
However, it turns out that this model is not well suited for the minimization of non-smooth functions, as we will see by means of
the one-dimensional counterexample below. The essential drawback of the model in \eqref{eq:wrongmod}
is that it does not account for neighboorhood information. Thus, it is rather natural to consider
the following model function:
\begin{equation}\label{eq:correctmod}
 \phi(x, \Delta; d) := \max_{g\in \UU(x, \Delta)} \dual{g}{d}
 \quad \text{with}\quad
 \UU(x,\Delta) := \bigcup_{\xi \in B_\Delta(x)} \partial f(\xi).
\end{equation}
A similar model based on the $\varepsilon$-subdifferential as defined in \cite{goldstein} is used in \cite{ayp} in the context of a non-smooth trust-region method.
If $f$ is Bouligand-differentiable and semi-smooth, then one can verify the conditions
in Assumption~\ref{assu:model}(\ref{it:model}) for the model function in \eqref{eq:correctmod} so that the above
convergence analysis applies. The proof thereof is analogous to the ones
of Lemma~\ref{lem:remainder} and \ref{lem:statmeas} below and therefore omitted.
Of course, the model function $\phi$ is much more costly compared to $\widetilde\phi$,
but the following counterexample shows that the simple model in \eqref{eq:wrongmod} might not
suffice. For this purpose, let us define
\begin{equation}\label{eq:objbsp}
 f: \R \to \R, \quad  f: x \mapsto \max\{-a x , -b x , x - (1+b)\},
\end{equation}
where $0 < b < a < \infty$ are given constants. This piecewise affine function is
trivially convex and admits two kinks at $x=0$ and $x=1$. If one applies the trust-region method
with the two different models to this function, the following lemma is obtained. Its
proof is not difficult, but rather technical and therefore postponed to Appendix~\ref{sec:counterex1d}.

\begin{lemma}\label{lem:counterex1d}
 Assume that the parameters and initial values in step~\ref{it:init} of the algorithm satisfy
 \begin{alignat}{3}
  & \beta_1 + \beta_1 \beta_2 < 1, & \quad
  & \eta_1 \geq \Big(\frac{b}{a} - 1\Big)\frac{\beta_1}{\beta_1 \beta_2-1}
  + \frac{b}{a} \label{eq:params}\\
  & x_0 \in \Big(\Big(1 + \frac{\beta_1 \beta_2 - 1}{\beta_1}\Big)^{-1},0\Big), & \quad
  & \Delta_{\min} > \Delta_0 := \frac{\beta_1 \beta_2 - 1}{\beta_1}\,x_0. \label{eq:delta0}
 \end{alignat}
 Then, the sequence of iterates generated by the trust-region algorithm
 performed with the model function $\widetilde\phi$ and $H_k = 0$ converge to $0$, which is \emph{not}
 stationary in any sense (in particular neither Clarke- nor Bouligand-stationary).

 By contrast, if one uses the model function $\phi$ from \eqref{eq:correctmod} instead, then the iterates
 converges to the global minimum at $x = 1$, no matter how the parameters and initial values are chosen.
\end{lemma}

\begin{remark}
 We emphasize that the failure of the trust-region method in case of the model in \eqref{eq:wrongmod}
 is not caused by the distinction of cases contained in Algorithm~\ref{alg:TR}.
% In fact, if one uses $\widetilde\phi$ as model function, then there is actually no more difference
% between the two trust-region subproblems \eqref{eq:qk} and \eqref{eq:tildeqk}, since the
% probability that the algorithm meets one of the kinks at $x=0$ and $x=1$ is zero and,
% in any other point, the objective $f$ is continuously differentiable so that $\partial f$
% is single valued there. Moreover, for the model function in \eqref{eq:wrongmod}, one readily verifies
% $\psi(x, \Delta) = |f'(x)|$ in all smooth points.
% Since in addition $\Delta_k \leq 1$ for all $k \in \N_0$
% (actually $(\Delta_k)$ is a null sequence as shown in Appendix~\ref{sec:counterex1d}),
% the computation of the quality indicator in \eqref{eq:rhomod} boils down to \eqref{eq:rho}.
% Therefore, in both cases $\Delta_k \geq \Delta_{\min}$ and $\Delta_k < \Delta_{\min}$,
% the iteration is the same and thus, Algorithm~\ref{alg:TR} turns into a standard (non-smooth)
% trust-region iteration.
 It is easy to see that, in both cases $\Delta_k \geq \Delta_{\min}$ and $\Delta_k < \Delta_{\min}$,
 the iteration is the same (unless the algorithm meets one of the kinks)
 and thus, Algorithm~\ref{alg:TR} turns into a standard (non-smooth) trust-region iteration.

 In our opinion, it is remarkable that this one-dimensional counterexample
 shows that a method based on a local model, which does not account for any neighborhood information,
 fails to converge even in case of a convex and piecewise affine objective.
 Of course, this observation is not new and we exemplarily refer to \cite[Section~5.7]{noll},
 where a similar two-dimensional example is discussed.
 However, if the initial radius $\Delta_0$ is chosen slightly different from the setting
 in \eqref{eq:delta0}, then the trust-region algorithm with the model in \eqref{eq:wrongmod}
 will converge to the global minimum at $x=1$. This indicates that, in frequent cases,
 it is not necessary to use an involved model of the form \eqref{eq:correctmod}, while
 simpler local models will suffice. Our algorithmic approach accounts for this observation
 by incorporating the distinction of cases $\Delta_k \gtrless \Delta_{\min}$ into
 the algorithm.
\end{remark}

%%%%%%%%%%%%%%%%%%%%%%%%%%%%%%%%%%%%%%%%%%%%%%%%%%%%%%%%
\section{Composite Functions}\label{sec:composite}

Our aim is to apply Algorithm~\ref{alg:TR} to (discretized) optimal control problems with non-smooth constraints.
In order to conform with the standard notation in optimal control, we denote the optimization
variable by $u$ from now on. Although this causes a slight abuse of notation, we tacitly replace
$x$ by $u$, when referring to the results of Section~\ref{sec:TRgeneral}.
Our general optimal control problem reads as follows:
\begin{equation}\label{eq:optctrl}
 \min_{u\in \R^n} \; f(u) := J(S(u), u),
\end{equation}
where $J: \R^m \times \R^n \to \R$, $m,n\in \N$ is continuously differentiable and $S: \R^n \to \R^m$ is assumed to
be directionally differentiable and locally Lipschitz continuous.
Note that $S$ is Bouligand-differentiable by \cite[Thm.~3.1.2]{scholtes2012introduction}.
In all what follows, we will frequently abbreviate $y := S(u) \in \R^m$.
Given $u\in \R^n$ and $\Delta > 0$, we suppose that we can construct an approximation
of the Bouligand-subdifferential of $S$ satisfying the following

\begin{assumption}\label{assu:approxsubdiff}
 Given $u\in \R^n$ and $\Delta > 0$, the approximation $\GG(u,\Delta) \subset \R^{m\times n}$ of the
 Bouligand-subdifferential is supposed to fulfill the following conditions:\\
 For all $u \in \R^n$ and all $\Delta > 0$, there holds
 \begin{equation}\label{eq:obermenge}
  \bigcup_{\xi \in B_\Delta(u)} \partial_B S(\xi) \subseteq \GG(u,\Delta)
 \end{equation}
 and, if $(u_k, \Delta_k) \to (u, 0)$ with $0\notin \partial f(u)$, then
 \begin{equation}\label{eq:Gapprox}
  \dist(\GG(u_k, \Delta_k), \partial_B S(u)))
  = \sup_{G\in \GG(u_k, \Delta_k)} \inf_{W\in \partial_B S(u)} \|G - W\|_{\R^{m\times n}} \to 0
 \end{equation}
 is valid.
\end{assumption}

With this approximation at hand, we construct our model function as follows:
\begin{equation}\label{eq:modeloptctrl}
 \phi(u, \Delta; d) := \sup_{G \in \GG(u, \Delta)} \dual{G^\top \nabla_y J(y, u) + \nabla_u J(y, u)}{d}.
\end{equation}
This model function allows the following reformulation of the modified trust-region subproblem, which will be useful
for the realization of the algorithm in case of the concrete optimization problem in Section~\ref{sec:optvi}.
Its proof is straightforward and therefore omitted.

\begin{lemma}\label{lem:reform}
 With the model function in \eqref{eq:modeloptctrl},
 the modified trust-region subproblem \eqref{eq:tildeqk} from step~\ref{step:TRmod} is equivalent to the following linear quadratic problem
 in the sense that they admit the same (global) optima:
 \begin{equation}\tag{$\mathfrak{Q}_k$}\label{eq:qkmod}
  \left.
  \begin{aligned}
   \min_{\substack{\zeta \in \R \\ d\in \R^n}}  \quad & \mathfrak{q}_k(d,\zeta)
   := J(y_k, u_k) + \zeta + \frac{1}{2}\, d^\top H_k d\\
   \text{\textup{s.t.}}  \quad & \|d\| \leq \Delta_k, \\
   & \dual{g}{d} \leq \zeta \quad\forall\,  g\in \{G^\top \nabla_y J(y_k, u_k) + \nabla_u J(y_k, u_k) : G \in \GG(u_k, \Delta_k)\} .
  \end{aligned}
  \;\right\}
 \end{equation}
 In addition, if $\bar d_k$ is a global minimizer of \eqref{eq:tildeqk}, then $(\bar d_k, \bar\zeta_k)$ with
 $\bar\zeta_k = \phi(u_k, \Delta_k, ;\bar d_k)$ solves \eqref{eq:qkmod} so that
 $\tilde q_k(\bar d_k) = \mathfrak{q}_k(\bar d_k, \bar \zeta_k)$.

 Moreover, if $(d_k, \zeta_k)$ is feasible for \eqref{eq:qkmod} and satisfies
 \begin{equation}\label{eq:cauchyzeta}
  f(x_k) - \mathfrak{q}_k(d_k, \zeta_k) \geq
  \frac{\mu}{2}\,\psi(x_k, \Delta_k) \,\min\Big\{ \Delta_k, \frac{\psi(x_k, \Delta_k)}{\|H_k\|} \Big\},
 \end{equation}
 then $d_k$ fulfills the modified Cauchy-decrease condition in \eqref{eq:tildecauchy}, too.
\end{lemma}

\begin{remark}
 Note that the optimal solution $(\bar d_k, \bar\zeta_k)$
 of \eqref{eq:qkmod} satisfies the modified Cauchy-decrease condition \eqref{eq:cauchyzeta},
 since $\bar d_k$ does so by Lemma~\ref{lem:cauchyglob}
 and $\tilde q_k(\bar d_k) = \mathfrak{q}_k(\bar d_k, \bar \zeta_k)$.
\end{remark}

Next, we show that the model function in \eqref{eq:modeloptctrl} satisfies the conditions in Assumption~\ref{assu:model}.
To this end, we require the following

\begin{assumption}\label{assu:derivS}
 For every $u\in \R^n$ and every $h\in \R^n$, there exists a $G\in \partial_B S(u)$ so that $S'(u;h) = G\,h$.
\end{assumption}

There is a large class of functions satisfying this assumption such as for instance
semi-smooth functions, as the next lemma shows.

\begin{lemma}
 If $S:\R^n \to \R^m$ is Bouligand differentiable and semi-smooth,
 then Assumption \ref{assu:derivS} is fulfilled.
\end{lemma}

\begin{proof}
 Let $u$ and $h$ be given and $(t_n)$ be an arbitrary null sequence.
 For every $n\in \N$,  Rademacher's theorem implies the existence of $h_n$ such that
 \begin{equation}\label{eq:smoothapprox}
  u + t_n h_n \in \DD_S \quad \text{and} \quad \|h_n - h\| = \OO(t_n).
 \end{equation}
 The semi-smoothness of $S$ then implies
 \begin{equation}\label{eq:semismooth}
  \frac{S(u + t_n h_n) - S(u)}{t_n} - S'(u + t_n h_n)h_n \to 0.
 \end{equation}
 The local Lipschitz continuity of $S$ moreover yields the boundedness of $\{S'(u + t_n h_n)\}$
 so that there exists $G \in \R^{m\times n}$ with $S'(u + t_n h_n) \to G$.
 As $S$ is Bouligand differentiable, \eqref{eq:smoothapprox} in turn implies the convergence of the first addend in
 \eqref{eq:semismooth} to $S'(u;h)$. This establishes the claim.
\hfill\end{proof}

\begin{lemma}\label{lem:remainder}
 Let $(u_k, \Delta_k) \subset \R^n \times \R^+$ be a sequence such that $u_k \to \tilde u$ and $\Delta_k \to 0$.
 Then, the linearization error satisfies
 \begin{equation*}
  \limsup_{k\to\infty} \sup_{d\in B_{\Delta_k}(0)}
  \frac{J(S(u_k + d), u_k+d) - J(S(u_k), u_k) + \phi(u_k, \Delta_k; d)}{\Delta_k}
  \leq 0
 \end{equation*}
 such that the model given by \eqref{eq:modeloptctrl} fulfills Assumption~\ref{assu:model}(\ref{it:remainder}).
\end{lemma}

\begin{proof}
 Let $u\in \R^n$, $\Delta > 0$, and $d\in B_\Delta(0)$ be arbitrary.
 By \cite[Prop.~3.1.1]{scholtes2012introduction} and the chain rule for Bouligand-differentiable functions, we have
 \begin{equation*}
 \begin{aligned}
  & J(S(u+d), u + d) - J(S(u),u)\\
  & =  \int_0^1 \dual{\nabla_y J(S(u+\theta d), u+ \theta d)}{S'(u+\theta d; d)}
  + \dual{\nabla_u J(S(u+\theta d), u+ \theta d)}{d}\,d\theta.
 \end{aligned}
 \end{equation*}
 By Assumption~\ref{assu:derivS}, for every $\theta \in [0,1]$, there exists a $G_\theta \in \partial_B S(u + \theta d)$
 such that $G_\theta\, h= S'(u+\theta d; d)$. This, together with \eqref{eq:obermenge},
 the definition of our model $\phi$, and $\|d\|\leq \Delta$, yields
 \begin{equation*}
 \begin{aligned}
  & J(S(u+d), u + d) - J(S(u),u)\\
  & = \int_0^1 \dual{G_\theta^\top\nabla_y J(S(u+\theta d), u+ \theta d)
  + \nabla_u J(S(u+\theta d), u+ \theta d)}{d}\,d\theta\\
  & \leq \phi(u, \Delta; d)\\[-1ex]
  %\int_0^1 \sup_{G\in \GG(u,\Delta)} \dual{G^\top\nabla_y J(S(u), u)} + \nabla_u J(S(u), u)}{d}\,d\theta\\
  &\quad + \sup_{G\in \cup_{\xi \in B_\Delta(u)} \partial_B S(\xi)} \|G\|\,
  \int_0^1 \|J'(S(u + \theta d), u + \theta d) - J'(S(u), u)\|\,d\theta \,\Delta.
 \end{aligned}
 \end{equation*}
 Now, let $(u_k, \Delta_k)$ be the sequence from the statement of the lemma and denote by
 $\tilde L > 0$ and $\widetilde\UU\subset \R^n$ the local Lipschitz constant of $S$ at $\tilde u$
 and the associated neighborhood of local Lipschitz continuity, respectively.
 Then, for $K\in \N$ sufficiently large, we have $B_{\Delta_k}(u_k) \subset \widetilde \UU$ and therefore
 \begin{equation*}
  \sup_{G\in \cup_{\xi \in B_{\Delta_k}(u_k)} \partial_B S(\xi)} \|G\| \leq \tilde L
  \quad \forall\, k \geq K.
 \end{equation*}
 Furthermore, the uniform continuity of $u \mapsto J'(S(u), u)$ on $\operatorname{cl}(\widetilde\UU)$
 and $(u_k, \Delta_k) \to (\tilde u, 0)$ imply
 \begin{equation*}
  \sup_{d \in B_{\Delta_k}(0)} \|J'(S(u_k + d), u_k + d) - J'(S(u_k), u_k)\|
  \to 0 \quad \text{as } k \to \infty.
 \end{equation*}
 Collecting all findings, we arrive at
 \begin{equation*}
 \begin{aligned}
  & \sup_{d\in B_{\Delta_k}(0)}
  \frac{J(S(u_k + d), u_k+d) - J(S(u_k), u_k) - \phi(u_k, \Delta_k; d)}{\Delta_k}\\
  &\qquad\qquad \leq \tilde L \, \sup_{d \in B_{\Delta_k}(0)}
  \int_0^1 \|J'(S(u + \theta d), u + \theta d) - J'(S(u), u)\|\,d\theta
  \to 0,
 \end{aligned}
 \end{equation*}
 which implies the assertion.
\hfill\end{proof}

\begin{lemma}\label{lem:statmeas}
 Let $\psi$ denote the stationarity measure from \eqref{eq:normersatz}, i.e.,
 \begin{equation*}
  \psi(u, \Delta)  := - \min_{\|h\|\leq 1} \phi(u, \Delta; h),
 \end{equation*}
 and $(u_k, \Delta_k)\subset \R^n\times \R^+$ so that
 \begin{equation}\label{eq:psitonull}
  u_k \to u, \quad \Delta_k \to 0, \quad\text{and}\quad \psi(u_k, \Delta_k) \to 0.
 \end{equation}
 Then $0\in \partial f(u)$ holds true such that the model from \eqref{eq:modeloptctrl} satisfies
 Assumption~\ref{assu:model}(\ref{it:cstat}). Herein $f$ again denotes the reduced objective, i.e., $f(u) = J(S(u), u)$.
\end{lemma}

\begin{proof}
 We argue by contradiction and assume that there is $\varepsilon > 0$ so that
 \begin{equation}\label{eq:distsubdiff}
  \dist(0, \partial f(u)) \geq \varepsilon.
 \end{equation}
 Let us denote the set of points, where $S$ and $f$ are differentiable by $\DD_S$ and $\DD_f$, respectively.
 Since $J$ is continuously differentiable, the chain rule implies $\DD_S \subseteq \DD_f$
 and, by Rademacher's theorem, we have $\lambda^n(\DD_f\setminus \DD_S) = 0$.
 Therefore, \cite[Thm.~2.5.1]{clarke} and the continuous differentiability of $J$ imply
 \begin{equation*}
 \begin{aligned}
  & \{G^\top \nabla_y J(y, u) + \nabla_u J(y, u) : G \in \partial_B S(u)\} \\
  & \subset \clos{\conv\big(g\in \R^n : \, \exists\, (u_n)\subset \DD_S :
  u_n \to u, \, S'(u_n)^\top\nabla_y J(y_n, u_n) + \nabla_u J(y_n, u_n) \to g \big)}\\
  & = \partial f(u)
 \end{aligned}
 \end{equation*}
 Therefore, due to Assumption~\eqref{eq:Gapprox}, there exist $K\in \N$ such that, for all $k\geq K$, it holds
 \begin{equation*}
 \begin{aligned}
  & \{G^\top \nabla_y J(y_k, u_k) + \nabla_u J(y_k, u_k) : G \in \GG(u_k, \Delta_k)\} \\
  & \quad \subset
  \{G^\top \nabla_y J(y, u) + \nabla_u J(y, u) : G \in \partial_B S(u)\} + B_{\varepsilon/2}(0)
  \subset \partial f(u) + B_{\varepsilon/2}(0).
 \end{aligned}
 \end{equation*}
 Since $\partial f(u)$ is convex, this in combination with \eqref{eq:distsubdiff}
 implies $\dist(\CC_k, 0) \geq \varepsilon/2$, where we abbreviated
 \begin{equation*}
  \CC_k := \clos{\conv\big(\{G^\top \nabla_y J(y_k, u_k) + \nabla_u J(y_k, u_k) :
  G \in \GG(u_k, \Delta_k)\}\big)}.
 \end{equation*}
 Next, let us define $\bar g$ as unique solution of the following VI:
% \begin{equation}\label{eq:bargdef}
%  \bar g := \argmin_{g\in \CC_K} \|g\|^2,
% \end{equation}
 \begin{equation*}
  \bar g\in \CC_k, \quad \dual{\bar g}{g - \bar g} \geq 0 \quad \forall g\in \CC_k.
 \end{equation*}
 Using this VI in combination with $\dist(\CC_k, 0) \geq \varepsilon/2$ results in
 \begin{equation*}
 \begin{aligned}
  \psi(u_k, \Delta_k)
  & = \max_{\|h\|\leq 1}\Big(\inf_{G\in \GG(u_k, \Delta_k)}
  \dual{G^\top \nabla_y J(y_k, u_k) + \nabla_u J(y_k, u_k)}{-h}\Big)\\
  & \geq \inf_{G\in \GG(u_k, \Delta_k)} \sdual{G^\top \nabla_y J(y_k, u_k) + \nabla_u J(y_k, u_k)}{\frac{\bar g}{\|\bar g\|}} \\
  & \geq \inf_{g\in \CC_k} \frac{\dual{g}{\bar g}}{\|\bar g\|} \geq  \|\bar g\|
  \geq \frac{\varepsilon}{2} \qquad \forall\, k \geq K.
 \end{aligned}
 \end{equation*}
 This however contradicts the last assumptions in \eqref{eq:psitonull}, which gives the claim.
\hfill\end{proof}

We collect the findings of this section in the following

\begin{corollary}\label{cor:convcomp}
 Under Assumptions~\ref{assu:approxsubdiff} and \ref{assu:derivS}, every accumulation point of the sequence of iterates
 generated by our non-smooth trust-region algorithm applied to \eqref{eq:optctrl} with the model function in \eqref{eq:modeloptctrl}
 is a C-stationary point of \eqref{eq:optctrl}.
\end{corollary}

\begin{proof}
 As shown in the above lemmata, Assumption~\ref{assu:model} is fulfilled, provided that Assumptions \ref{assu:approxsubdiff} and \ref{assu:derivS} hold true.
 Therefore, Theorem~\ref{thm:TRconv} gives the assertion.
\hfill\end{proof}

%%%%%%%%%%%%%%%%%%%%%%%%%%%%%%%%%%%%%%%%%%%%%%%%%%%%%%%%%%%%%%%%%%%%%%%%%%%%
\section{Optimization of Variational Inequalities of the Second Kind}\label{sec:optvi}

We now focus on the following class of nonsmooth optimization problems:
\begin{equation}\tag{P$_{\textup{VI}}$}\label{eq:piv}
 \left.
 \begin{aligned}
  \min_{y,u\in \R^n} & \quad J(y,u)\\
  \text{s.t.} & \quad \dual{A y}{v-y} + \|v\|_1 - \|y\|_1 \geq \dual{u}{v-y} \quad \forall\, v\in \R^n,
 \end{aligned}
 \qquad\right\}
\end{equation}
where $J: \R^n \times \R^n \to \R$ is smooth, $A\in \R^{n\times n}$ is symmetric and positive definite, and
$\|.\|_1$ denotes the 1-norm, i.e., $\|v\|_1 = \sum_{i=1}^n |v_i|$. Note that the constraints are given in form of a variational inequality of the second kind.

In the next proposition we summarize some known results about \eqref{eq:piv}. For more details on this, we refer to \cite{delosreyesmeyer}.

\begin{proposition}\label{prop:vi}
 Let $u\in \R^n$ be given. Then there holds:
 \begin{itemize}
  \item There exists a unique solution $y\in \R^n$ of the VI in \eqref{eq:piv}, i.e.,
  \begin{equation}\tag{VI}\label{eq:vi}
   \dual{A y}{v-y} + \|v\|_1 - \|y\|_1 \geq \dual{u}{v-y} \quad \forall\, v\in \R^n.
  \end{equation}
  \item $y$ is the solution of \eqref{eq:vi}, iff there exists a $q\in \R^n$ such that
  \begin{equation}\label{eq:compl}
   A y + q = u, \quad y_i\, q_i = |y_i|, \quad |q_i| \leq 1, \quad i = 1, ..., n
  \end{equation}
  \item The solution mapping $S: \R^n \ni u \mapsto y \in \R^n$ is globally Lipschitz continuous
  and directionally differentiable.
  Its directional derivative $\eta = S'(u;h)$ at $u$ in direction $h\in \R^n$ is given by
  the unique solution of
  \begin{equation}\label{eq:rablvi}
   \eta \in \KK(y), \quad \dual{A\eta}{v - \eta} \geq \dual{h}{v-\eta} \quad \forall\, v\in \KK(y),
  \end{equation}
  where
  \begin{equation}\label{eq:critcone}
   \KK(y) := \{v\in \R^n: v_i = 0, \text{ if } |q_i|<1, \; v_i\,q_i \geq 0, \text{ if } y_i = 0
   \wedge |q_i| = 1\}.
  \end{equation}
 \end{itemize}
\end{proposition}

Thanks to these properties, we may formulate problem \eqref{eq:piv} in reduced form as
\begin{equation}
  \min_{u \in \R^n} \; f(u):=J(S(u),u),
\end{equation}
so that a problem of the form \eqref{eq:optctrl} is obtained. Our aim in the following is to verify the hypotheses on the
general problem \eqref{eq:optctrl}, i.e., Assumptions~\ref{assu:approxsubdiff} and \ref{assu:derivS}.
For this purpose, we first have to charactrize the Bouligand-subdifferential of $S$, which is addressed in the next subsection.

%%%%%%%%%%%%%%%%%%%%%%%%%%%%%%%%%%%%%%%%%%%%%%%%%%%%%%%%%%%%%%%%%%%%%%%%%%%%%%%%
\subsection{Characterization of the Bouligand-Subdifferential}

Given $u\in \R^n$ with $y = S(u)$, we define the following sets
\begin{subequations}\label{eq:sets}
\begin{alignat}{2}
 \AA &:= \{i \in \{1, ..., n\} : y_i = 0\} & \quad & \text{(active set)} \label{eq:active}\\
 \AA_s &:= \{i \in \{1, ..., n\} : |q_i| < 1 \} & & \text{(stongly active set)} \\
 \II &:= \{i \in \{1, ..., n\} : y_i \neq 0\} & & \text{(inactive set)} \label{eq:inactive}\\
 \BB &:= \{i \in \{1, ..., n\} : y_i = 0 \wedge |q_i| = 1\} & & \text{(biactive set)}.
\end{alignat}
\end{subequations}
Note that these sets depend on $y$ and thus indirectly on $u$
so that it would be more appropriate to write $\AA(y)$ or $\AA(u)$ etc., but,
for the sake of readability, we suppress this dependency throughout this subsection.
This will be different in Section~\ref{sec:approxdS}, where we have to distinguish between the active sets in different points.
Note that, because of the complementarity like relation in \eqref{eq:compl}, one has
$\AA_s \subset \AA$, and therefore $\AA = \AA_s \cup \BB$.

\begin{lemma}\label{lem:Sglatt}
 $S$ is differentiable at $u$ iff $\KK(y) = \{v\in \R^n: v_i = 0, \text{ if } y_i = 0\}$.
\end{lemma}

\begin{proof}
 It is clear that, if $\KK(y)$ takes the form stated in the Lemma, then $\KK(y)$ is a linear subspace
 and, as a convex projection on a linear subspace, $S'(u,.)$ is a linear mapping so that
 $S$ is differentiable at $u$.

 To show the converse assertion, we first show that
 \begin{equation}\label{eq:rablsurj}
  S'(u;\R^n) = \KK(y).
 \end{equation}
 By \eqref{eq:rablvi}, we already have $S'(u;\R^n) \subset \KK(y)$. To see the reverse inclusion,
 let $z\in \KK(y)$ be arbitrary and set $h = A\,z$.
 Then we trivially obtain $\dual{Az}{v-z} = \dual{h}{v-z}$ for all $v\in \KK(y)$ so that $z = S'(u;h)$,
 which shows \eqref{eq:rablsurj}.
 Moreover, if $S$ is differentiable so that  $h \mapsto S'(u;h)$ is linear, then $S'(u;\R^n)$ becomes
 a linear subspace and, by \eqref{eq:rablsurj}, so does $\KK(y)$.
 Therefore $v\in \KK(y)$ implies $-v\in \KK(y)$, which, due to \eqref{eq:critcone} yields
 \begin{equation*}
  0\leq v_i\,q_i \leq 0 \quad \forall\, i\in \BB = \{j \in \{1, ..., n\} : y_j = 0 \wedge |q_j| = 1\}.
 \end{equation*}
 Since $q_i \neq 0$ in $\BB$, this yields $v_i = 0$ in $\BB$, which, together with $v_i = 0$ in $\AA_s$,
 see \eqref{eq:critcone}, finally gives $v_i = 0$ in $\BB \cup \AA_s = \AA$ as claimed.
\hfill\end{proof}

Now, we are in the position to give a precise characterization of the Bouligand-subdifferential of $S$.
To this end, we introduce the following

\begin{definition}
 Let $\NN \subset \{1, ..., n\}$ be an index set.
 Then we define the matrices $A(\NN) \in \R^{n\times n}$ and $\chi(\NN)\in \R^{n\times n}$ by
 \begin{equation*}
  A(\NN)_{ij} :=
  \begin{cases}
   A_{ij}, & \text{if } i,j \in \{1, ..., n\}\setminus \NN,\\
   0, & \text{if } i \vee j \in \NN, i\neq j,\\
   1, & \text{if } i=j \in \NN,
  \end{cases}
   \quad
  \chi(\NN)_{ij} :=
  \begin{cases}
   1, & i=j \in \{1, ..., n\}\setminus \NN,\\
   0, & \text{otherwise}.
  \end{cases}
 \end{equation*}
\end{definition}

\begin{theorem}\label{thm:boulimatrix}
 Let $u\in \R^n$ be fixed, but arbitrary, and let $y = S(u)$. Then there holds
 \begin{equation}\label{eq:boulimatrix}
  \begin{aligned}
  \partial_B S(u)
  = \{ A(\AA_s\cup\BB_0)^{-1}\chi(\AA_s\cup\BB_0) :  \BB_0\subseteq \BB \}.
 \end{aligned}
 \end{equation}
\end{theorem}

\begin{remark}
 Note that $A(\NN)$ is indeed invertible for every index set $\NN \subset \{1, ..., n\}$, since it is positive definite:
 For an arbitrary $v\in \R^n$, we obtain
 \begin{equation*}
 \begin{aligned}
  v^\top A(\NN) v
  &= \sum_{i,j\notin \NN} A_{ij} v_i v_j + \sum_{i\in \NN} v_i^2\\
  &= [(I - \chi(\NN))v]^\top A\, (I - \chi(\NN))v + \sum_{i\in \NN} v_i^2
  \geq \min\{\lambda_{\min}, 1\} \|v\|^2,
 \end{aligned}
 \end{equation*}
 where $\lambda_{\min} > 0$ denotes the minimal eigenvalue of $A$.
\end{remark}

\begin{remark}
 We could equivalently replace the last line in the definition of $\AA(\NN)$ by
 \begin{equation*}
  A(\NN)_{ij} := c, \quad \text{if } i=j\in \NN
 \end{equation*}
 with some $c\neq 0$, since no matter, which value is chosen for $c\neq 0$, the matrix $A(\NN)^{-1}\chi(\NN)$ is always the same, as
 \begin{equation}\label{eq:Beq}
  A(\NN) \tilde\eta = \chi(\NN) h\quad
  \Longleftrightarrow \quad
  \left\{
  \begin{aligned}
   \tilde\eta_i &= 0 & & \forall \, i\in \NN,\\
   \sum_{j \notin \NN} A_{ij}\tilde \eta_j &= h_i & & \forall\, i\in \{1, ..., n\} \setminus\NN,
  \end{aligned}
  \right.
 \end{equation}
 and there is no more $c$ appearing on the right hand side of this equivalence.
\end{remark}

\begin{proofof}{Theorem \ref{thm:boulimatrix}}
 Recall that
 \begin{equation}\label{eq:subdiffRn}
  \partial_B S(u) = \big\{B\in \R^{n\times n} :
  \exists \, \{u_n\}\subset \DD_S \text{ with } u_n \to u,\, S'(u_n) \to B \big\},
 \end{equation}
 where $\DD_S$ again denotes the (dense) set of points, where $S$ is differentiable.
 Now consider an arbitrary $B\in \partial_B S(u)$ so that there is a sequence in $\DD_S$ with
 \begin{equation}\label{eq:convjacobi}
  u_n \to u \quad \text{and} \quad S'(u_n) \to B.
 \end{equation}
 The Lipschitz continuity of $S$ implies
 \begin{equation}\label{eq:convstateslack}
  y_n := S(u_n) \to S(u) =: y \quad \Longrightarrow \quad
  q_n = u_n - A y_n \to u - A y = q.
 \end{equation}
 Let us denote the active set associated with $y_n$ by $\AA^n$ and analogously for $\II^n$ etc.
 Then, from \eqref{eq:convstateslack}, we deduce the existence of an $N\in \N$ such that
 \begin{equation}\label{eq:convset}
  \II \subset \II^n \quad \text{and}\quad \AA_s \subset \AA_s^n \quad
  \quad \forall\, n\geq N.
 \end{equation}
 Next let $h\in \R^n$ be fixed, but arbitrary.
 Since $u_n\in \DD_S$, we know from Lemma \ref{lem:Sglatt} that $\eta_n := S'(u_n)h$ solves
 \begin{equation}\label{eq:jacobieq}
  \eta_i^n = 0 \quad \forall\, i\in \AA^n, \quad
  \sum_{j=1}^n A_{ij} \eta_j^n = h_i \quad\forall\, i\in \II^n.
 \end{equation}
 By \eqref{eq:convjacobi} we obtain that
 \begin{equation}\label{eq:conveta}
  \tilde\eta:= B\,h = \lim_{n\to \infty} \eta_n.
 \end{equation}
 Therefore, from \eqref{eq:convset}--\eqref{eq:conveta} it follows that
 \begin{equation}
  \tilde\eta_i = 0 \quad \forall\, i\in \AA_s, \quad
  \sum_{j=1}^n A_{ij} \tilde\eta_j = h_i \quad\forall\, i\in \II.
 \end{equation}
 It remains to investigate what happens on the biactive set $\BB = \AA\setminus \AA_s$.
 For this purpose we introduce
 \begin{equation*}
  \BB_0 := \{i\in \BB : \exists \text{ a subsequence } \{n_k\} \text{ such that }
  y_i^{n_k} = 0 \;\forall\, k\in \N\}
 \end{equation*}
 so that, for all $i\in \BB\setminus\BB_0$, it holds that $y^n_i \neq 0$ for all $n\in \N$ sufficiently large.
 Then we deduce from \eqref{eq:jacobieq} that $\eta^{n_k}_i = 0$ for all $i \in \BB_0$ and all $k\in \N$ and that
 $\sum_{j=1}^n A_{ij} \eta^n_j = h_i$ for all $i\in \BB\setminus\BB_0$, provided that $n\in \N$ is sufficiently large.
 Since $\eta_n \to \tilde\eta$, we obtain in this way
 \begin{equation}\label{eq:tildeeta}
  \tilde\eta_i = 0 \quad \forall\, i\in \AA_s \cup \BB_0, \quad
  \sum_{j\notin \AA_s \cup \BB_0} A_{ij} \tilde\eta_j = h_i
  \quad\forall\, i\in \{1, ..., n\}\setminus (\AA_s \cup \BB_0).
 \end{equation}
 Thus, in view of \eqref{eq:Beq} and since $h$ was arbitrary, we observe that
 \begin{equation}\label{eq:subdiffmatrix}
  B = A(\AA_S\cup\BB_0)^{-1}\chi(\AA_s\cup\BB_0).
 \end{equation}
 Hence, $B$ has indeed the form stated in the theorem.

 To complete the proof, we need to show that, for every subset $\BB_0\subset \BB$ the corresponding
 matrix $B$ given by \eqref{eq:boulimatrix} is an element of $\partial_B S(u)$.
 To this end, let $\BB_0\subset \BB$ be arbitrary, but fixed and let us abbreviate
 $\BB_1 := \BB\setminus\BB_0$. In the following, we show that there exist a sequence $\{u_n\}$ satisfying
 \begin{equation}\label{eq:sequn}
 \begin{gathered}
  u_n \in D_S,\quad y^n_i = 0\quad \forall\, i \in \BB_0, \quad y^n_i \neq 0\quad \forall\, i \in \BB_1,
  \quad \forall\, n\in \N,\\
  \text{and} \quad u_n \to u, \quad S'(u_n)\to B \quad \text{as } n\to\infty,
 \end{gathered}
 \end{equation}
 which, according to \eqref{eq:subdiffRn} implies $B\in \partial_S S(u)$.
 To verify the existence of such a sequence, let $\varepsilon > 0$ and define
 \begin{equation*}
  y^\varepsilon := y + \sum_{k\in \BB_1} \varepsilon \, \sgn(q_{k})\,\mathrm{e}_{k},
 \end{equation*}
 where $\mathrm{e}_i$ denotes the $i$-the Euclidian unit vector. By construction we obtain
 for the inactive and active set associated with $y^\varepsilon$ that
 \begin{equation}\label{eq:Ieps}
  \II^\varepsilon = \II \cup \BB_1 \quad \text{and}\quad \AA^\varepsilon = \AA\setminus\BB_1.
 \end{equation}
 Moreover, we set
 \begin{equation*}
  q^\varepsilon = q - \sum_{k\in \BB_0} \varepsilon \,\sgn(q_k) \,\mathrm{e}_k.
 \end{equation*}
 Thus, for $\varepsilon \in (0,1]$,
 we obtain $|q^\varepsilon_i| \leq 1$ for all $i = 1, ..., n$. Moreover, the above construction leads to
 \begin{equation}\label{eq:Aseps}
  \AA_s^\varepsilon = \AA_s \cup \BB_0 = \AA \setminus \BB_1,
 \end{equation}
 which, together with \eqref{eq:Ieps}, shows that
 \begin{equation}\label{eq:nobiactive}
  \BB^\varepsilon = \AA^\varepsilon \setminus\AA_s^\varepsilon = \emptyset,
 \end{equation}
 i.e., the biactive set associated with $y_\varepsilon$ is empty.
 Furthermore, if we define
 \begin{equation}\label{eq:ueps}
  u^\varepsilon := u + \varepsilon \sum_{k\in \BB_0} \sgn(q_k)\,A\,\mathrm{e}_k
  + \varepsilon \sum_{k\in \BB_1} \sgn(q_k) \,\mathrm{e}_k,
 \end{equation}
 then we obtain by construction that
 \begin{equation*}
  A\, y^\varepsilon + q^\varepsilon = u^\varepsilon,
  \quad y_i^\varepsilon\, q_i^\varepsilon = |y_i^\varepsilon|,
  \quad |q_i^\varepsilon| \leq 1, \quad i = 1, ..., n,
 \end{equation*}
 which, due to \eqref{eq:compl}, implies $y^\varepsilon = S(u^\varepsilon)$.
 Because of \eqref{eq:nobiactive}, we further have
 $\KK(y^\varepsilon) = \{v\in \R^n: v_i = 0, \text{ if } y^\varepsilon_i = 0\}$, which, thanks to
 Lemma \ref{lem:Sglatt}, in turn implies $u^\varepsilon \in D_s$. In addition,
 \eqref{eq:ueps} immediately gives $u_\varepsilon \to u$ as $\varepsilon\searrow 0$.
 Because of \eqref{eq:Ieps}, we moreover have $y^\varepsilon_i \neq 0$
 for all $i\in \BB_1$ and, due to complementarity and \eqref{eq:Aseps}, $y^\varepsilon_i = 0$
 for all $i\in \BB_0$. Therefore, the sequence $\{u^\varepsilon\}_{\varepsilon>0}$ almost satisfies
 all conditions required in \eqref{eq:sequn}, except $S'(u^\varepsilon) \to B$.
 To establish this, let $\{\varepsilon_n\}_{n\in \N}$ be a sequence tending to zero and denote the
 associated $u^{\varepsilon_n}$ simply by $u_n$. The global Lipschitz continuity of $S$ yields that
 \begin{equation*}
  \|S'(u_n)\|_{\R^{n\times n}} \leq L \quad \forall\, n\in \N,
 \end{equation*}
 where $L>0$ is the Lipschitz constant of $S$. Thus there exists a convergent subsequence, i.e.,
 \begin{equation*}
  S'(u_{n_k}) \to \tilde B \quad\text{as } k \to \infty.
 \end{equation*}
 Since $y_i^{n_k} = 0$ for all $i \in \BB_0$ and $y_i^{n_k} \neq 0$ for all $i\in \BB_1$,
 we can argue completely analogously to the first part of the proof to show
 \begin{equation*}
  \tilde B = A(\AA_S\cup\BB_0)^{-1}\chi(\AA_s\cup\BB_0) = B,
 \end{equation*}
 which finally establishes the claim.
\hfill\end{proofof}

\begin{lemma}\label{lem:Srabl}
 For all $u,h \in \R^n$, there exists $G\in \partial_B S(u)$ such that $S'(u;h) = G\,h$.
 Hence Assumption~\ref{assu:derivS} is fulfilled by the control-to-state map of \eqref{eq:piv}.
\end{lemma}

\begin{proof}
 Let $u, h \in \R^n$ be arbitrary and again set $y = S(u)$. As above we denote by $\AA_s$, $\II$,
 and $\BB$ the active, inactive, and bi-active sets associated with $y$.
 Recall that the directional derivative of the solution mapping
 in direction $h$ is given by the unique solution of
 \begin{equation} \label{eq:derivative VI}
  \eta \in \KK(y), \quad
  \dual{A \eta}{v-\eta} \geq \dual{h}{v-\eta} \quad \forall\, v\in \KK(y),
 \end{equation}
 with $\KK(y)$ as defined in \eqref{eq:critcone}.
 This cone can equivalently be expressed as
 \begin{equation*}
  \KK(y)=\left\{ v \in \R^n: v_i=0, \text{ if }|q_i|<1, \; v_i
  \begin{cases}
   \geq 0, & \text{if } y_i=0,\, q_i=1 \\
   \leq 0, & \text{if } y_i=0,\, q_i=-1
  \end{cases}
  \right\}\,,
 \end{equation*}
 which in turn leads to the following equivalent expression for $\eta$
 \begin{equation}\label{eq:etaequiv}
  \eta_i =
  \begin{cases}
   \max\{0,(I-A) \eta + h\}_i, & \text{ if } y_i=0,\, q_i=1\\
   0, & \text{ if } |q_i|<1\\
   \min\{0,(I-A) \eta + h\}_i, & \text{ if } y_i=0,\, q_i=-1\\
   ((I-A) \eta +h)_i, & \text{ elsewhere}.
  \end{cases}
 \end{equation}
 So, if we define
 \begin{equation*}
  \BB_0 := \{ i\in \{1, ..., n\} :
  y_i=0,\, |q_i| = 1,\, q_i ((I-A) \eta + h)_i < 0\} \subseteq \BB,
 \end{equation*}
 then a comparison of \eqref{eq:etaequiv} with \eqref{eq:Beq} shows that
 \begin{equation*}
  \eta = A(\AA_s \cup \BB_0)^{-1}\chi(\AA_s \cup \BB_0) h.
 \end{equation*}
 Since the matrix on the right hand side is an element of $\partial_B S(u)$, this
 establishes the assertion.
\hfill\end{proof}

%%%%%%%%%%%%%%%%%%%%%%%%%%%%%%%%%%%%%%%%%%%%%%%%%%%%%%%%%%%
\subsection{Approximation of the Bouligand-Subdifferential}\label{sec:approxdS}

The aim of the upcoming section is to construct a \emph{computable} approximation of $\partial_B S$, that satisfies the
conditions in Assumption~\ref{assu:approxsubdiff} so that the model function given by \eqref{eq:modeloptctrl}
fulfills Assumption~\ref{assu:model} for the convergence result in Theorem~\ref{thm:TRconv}.
For this purpose, we need a sharpened Lipschitz continuity result for the solution operator $S$ associated with \eqref{eq:vi}:

\begin{lemma}\label{lem:Slip}
 For all $u_1, u_2 \in \R^n$, there holds
 \begin{equation*}
 \begin{aligned}
  \|y_1 - y_2\|_\infty &\leq L_y\, \|u_1 - u_2 \|  & & \text{with }  L_y = \frac{1}{\lambda_{\min}},\\
  \|q_1 - q_2\|_\infty &\leq L_q\, \|u_1 - u_2 \|  & & \text{with }  L_q = \frac{\lambda_{\max}}{\lambda_{\min}} + 1,
 \end{aligned}
 \end{equation*}
 where $y_i, q_i \in \R^n$, $i=1,2$, are the solutions of \eqref{eq:compl} associated with $u_i$, and
 $\lambda_{\min}$ and $\lambda_{\max}$ denote the minimal and maximal eigenvalue of $A$, respectively.
\end{lemma}

\begin{proof}
 By testing the VI for $y_1$ with $y_2$ and vice versa and adding the arising inequalities, we obtain
 \begin{equation}\label{eq:ylip}
  \lambda_{\min} \|y_1 - y_2\| \leq \|u_1 - u_2\|,
 \end{equation}
 which immediately gives the first assertion. The second directly follow from the first equation in \eqref{eq:compl}, which yields
 \begin{equation*}
  \|q_1 - q_2\| \leq \|A\|_{\R^{n\times n}} \|y_1 - y_2\| + \|u_1 - u_2\|.
 \end{equation*}
 Inserting \eqref{eq:ylip} then implies the second estimate in the statement of the lemma.
\hfill\end{proof}

As the active, inactive, and biactive sets at multiple points will occur in what follows, we denote these sets
for a given $u\in \R^n$ by $\AA_s(u)$, $\II(u)$, and $\BB(u)$. (Note that these sets are determined by $y$ and $q$,
which in turn uniquely depend on $u$.)

\begin{definition}\label{def:possbiact}
 Let $u\in \R^n$ and $\Delta > 0$ be given and set $y = S(u)$. Then we define the set of \emph{possibly biactive indices} by
 \begin{equation*}
 \begin{aligned}
  \PP(u, \Delta) := \{i \in \{1, ..., n\} :  |y_i| < L_y \,\Delta \; \wedge \;
  \big||q_i| - 1\big| < L_q \,\Delta\}.
%  \PP(u, \Delta) := \BB(u) & \cup \{i \in \{1, ..., n\} : i \in \II(u),\, |y_i| \leq L_y \,\Delta\} \\
%  & \cup \{i \in \{1, ..., n\} : i \in \AA_s(u),\, \big||q_i| - 1\big| \leq L_q \,\Delta\}.
 \end{aligned}
 \end{equation*}
\end{definition}

In view of Lemma~\ref{lem:Slip}, it is clear that
\begin{equation}\label{eq:subdiffinkl}
 \bigcup_{\xi \in B_\Delta(u)} \BB(\xi) \subset \PP(u, \Delta),
\end{equation}
which is essential for the upcoming analysis. Given the set of possibly active indices, we construct our approximation of
the Bouligand-subdifferential as follows:
\begin{equation}\label{eq:subdiffapprox}
 \GG(u, \Delta) := \{ A(\AA_s(u)\cup \BB_0)^{-1}\chi(\AA_s(u) \cup \BB_0) : \BB_0 \subseteq \PP(u, \Delta) \}.
\end{equation}
As an immediate consequence of \eqref{eq:subdiffinkl} and Theorem~\ref{thm:boulimatrix}, we obtain

\begin{lemma}\label{lem:obermenge}
 The approximation of the Bouligand-subdifferential in \eqref{eq:subdiffapprox} satisfies condition \eqref{eq:obermenge}.
\end{lemma}

On the other hand, we find the following:

\begin{lemma}\label{lem:Gapprox}
 Let $(u_k, \Delta_k) \subset \R^n \times \R^+$ be a sequence with
 $(u_k, \Delta_k) \to (u,0)$. Then, there exists an index $K\in \N$
 (depending on $u$) so that $\PP(u_k, \Delta_k) \subseteq \BB(u)$
 for all $k\geq K$. Therefore, for all $k\geq K$, there holds $\GG(u_k, \Delta_k) \subseteq \partial_B S(u)$
 such that condition \eqref{eq:Gapprox} is fulfilled, too.
\end{lemma}

\begin{proof}
 Again, we denote the state associated with the limit $u$ by $y = S(u)$. Moreover, we define
 \begin{equation*}
  \delta_y := \min_{i\in \II(u)} \, |y_i| > 0
  \quad \text{and}\quad \delta_q := \min_{i\in \AA_s(u)} \big||q_i| - 1\big| > 0.
 \end{equation*}
 Since $u_k \to u$ and $S$ is globally Lipschitz, there exists $K_1 \in \N$ so that
 \begin{equation*}
  \min_{i\in \II(u)} \, |y_i^k| \geq \frac{\delta_y}{2}
  \quad \text{and} \quad
  \min_{i\in \AA_s(u)} \, \big||q_i^k| - 1\big| \geq \frac{\delta_q}{2} \quad \forall\, k\geq K_1.
 \end{equation*}
 Moreover, as $\Delta_k \to 0$, we can find another index $K_2\in \N$ so that
 \begin{equation*}
  \Delta_k < \min\Big\{\frac{\delta_y}{2 L_y}, \frac{\delta_q}{2 L_q} \Big\}
  \quad \forall\, k\geq K_2.
 \end{equation*}
 Consequently, we obtain for all $i\in \II(u)$ that
 \begin{equation*}
  |y^k_i| \geq \frac{\delta_y}{2} > L_y\,\Delta_k
  \quad \Longrightarrow\quad i \notin \PP(u_k, \Delta_k)
  \quad \forall\, k\geq K :=\max\{K_1, K_2\}.
 \end{equation*}
 Analogously, for all $i\in \AA_s(u)$, we have
 \begin{equation*}
  \big| |q^k_i|-1 \big| > L_q\,\Delta_k
  \quad \Longrightarrow\quad i \notin \PP(u_k, \Delta_k)
  \quad \forall\, k\geq K :=\max\{K_1, K_2\},
 \end{equation*}
 and therefore, since $\BB(u) = \{1, ..., n\} \setminus (\AA_s(u) \cup \II(u))$, it follows
 $\PP(u_k, \Delta_k) \subset \BB(u)$ as claimed. The second assertion of the lemma then
 immediately follows from Theorem~\ref{thm:boulimatrix} and the construction of our approximation in
 \eqref{eq:subdiffapprox}.
\hfill\end{proof}

For convenience of the reader, we next state the precise algorithm that arises, when applying
Algorithm~\ref{alg:TR} to \eqref{eq:piv}. We again use the reduced objective function $f(\cdot) = J(S(\cdot), \cdot)$.

\begin{algorithm}[Trust-Region Algorithm for the solution of \eqref{eq:piv}] \label{alg:trvi}
 \begin{algorithmic}[1]
  \STATE Initialization:\\
  Choose constants
  \begin{equation*}
   \Delta_{\min} > 0, \quad 0 < \eta_1 < \eta_2 < 1, \quad 0 < \beta_1 < 1 < \beta_2, \quad 0 < \mu \leq 1
  \end{equation*}
  an initial value $u_0 \in \R^n$, and an initial TR-radius $\Delta_0 > \Delta_{\min}$. Set $k=0$.
  \FOR{k= 0, 1, 2, ...}
   \STATE Solve the variational inequality \eqref{eq:vi} to compute the state $y_k$ associated with the control $u_k$.
   \STATE Choose a subset $\BB_k \subseteq \BB(u_k)$, solve the \emph{adjoint equation}
   \begin{equation*}
    A(\AA_s(u_k) \cup \BB_k) p_k = \nabla_y J(y_k, u_k),
   \end{equation*}
   and set $g_k = p_k + \nabla_u J(y_k, u_k)$.
   \STATE Choose a matrix $H_k \in \R^{n\times n}_{\textup{sym}}$, e.g.\ via a BFGS-update using $g_k$.
   \IF{$g_k = 0$}
    \STATE STOP the iteration, $0 \in \partial_B f(u_k)$.
   \ELSE
    \IF{$\Delta_k > \Delta_{\min}$}
     \STATE Compute an inexact solution $d_k$ of the \emph{trust-region subproblem}
     \begin{equation}\tag{Q$_{k}$}
      \left.
      \begin{aligned}
       \min_{d\in \R^n} & \quad q_k(d) := f(u_k) + \dual{g_k}{d} + \frac{1}{2}\, d^\top H_k d\\
       \text{s.t.} & \quad \|d\| \leq \Delta_k,
      \end{aligned}
      \quad\right\}
     \end{equation}
     that satisfies the \emph{generalized Cauchy-decrease condition}
     \begin{equation*}
      f(u_k) - q_k(d_k) \geq
      \frac{\mu}{2}\,\|g_k\|\,\min\Big\{ \Delta_k, \frac{\|g_k\|}{\|H_k\|} \Big\}.
     \end{equation*}
     \STATE Compute the quality indicator
     \begin{equation*}
      \rho_k := \frac{f(u_k) - f(u_k + d_k)}{f(u_k) - q_k(d_k)}.
     \end{equation*}
    \ELSE
     \STATE Identify the possibly biactive indices and denote the elements of the
     powerset of $\PP(u_k, \Delta_k)$ by $\BB^k_1, ..., \BB^k_{m_k}$ with $m_k = 2^{|\PP(u_k, \Delta_k)|}$.
     \FOR{$j = 1, ..., m_k$}
      \STATE Solve the \emph{adjoint equation}
      \begin{equation*}
       A(\AA_s(u_k) \cup \BB^k_j) p^k_j = \nabla_y J(y_k, u_k),
      \end{equation*}
      and set $g^k_j = p^k_j + \nabla_u J(y_k, u_k)$.
     \ENDFOR
      \STATE Compute an inexact, but feasible solution $d_k$ of the \emph{modified trust-region subproblem}
      \begin{equation}\tag{$\mathfrak{Q}_k$}\label{eq:qkvimod}
       \left.
       \begin{aligned}
        \min_{\zeta \in \R, d\in \R^n}  \quad & \mathfrak{q}_k(d,\zeta)
        := f(u_k) + \zeta + \frac{1}{2}\, d^\top H_k d\\
        \text{\textup{s.t.}}  \quad & \|d\| \leq \Delta_k, \\
        & \dual{g^k_j}{d} \leq \zeta \quad\forall\,  j = 1, ..., m_k .
       \end{aligned}
       \quad\right\}
      \end{equation}
      that satisfies the \emph{modified Cauchy-decrease condition}
     \begin{equation}\label{eq:cauchyvi}
      f(x_k) - \mathfrak{q}_k(d_k, \zeta_k) \geq
      \frac{\mu}{2}\,\psi(u_k, \Delta_k) \,\min\Big\{ \Delta_k, \frac{\psi(u_k, \Delta_k)}{\|H_k\|} \Big\}.
     \end{equation}
     \STATE Compute the stationarity measure $\psi(u_k, \Delta_k)$ as solution of
     \begin{equation}\label{eq:statmeas}
      \psi(u_k, \Delta_k)
      = - \min_{\xi \in \R, d\in \R^n}\{ \xi : \|d\|\leq 1, \; \dual{g^k_j}{d} \leq \xi \quad\forall\,  j = 1, ..., m_k \}.
     \end{equation}
     \STATE Compute the modified quality indicator
     \begin{equation*}
      \rho_k :=
      \begin{cases}
       \displaystyle{\frac{f(u_k) - f(u_k + d_k)}{f(u_k) - \mathfrak{q}_k(d_k, \zeta_k)}},
       & \text{if }  \psi(u_k, \Delta_k) > \|g_k\|\, \Delta_k \\
       0, & \text{if }  \psi(u_k, \Delta_k) \leq \|g_k\|\, \Delta_k.
      \end{cases}
     \end{equation*}
    \ENDIF
    \STATE Update: Set
    \begin{align*}
     u_{k+1} & :=
     \begin{cases}
      u_k, & \text{if } \rho_k \leq \eta_1 \quad \text{(null step)},\\
      u_k + d_k, & \text{otherwise} \quad \text{(successful step)},
     \end{cases} \\
     \Delta_{k+1} &:=
     \begin{cases}
      \beta_1\,\Delta_k, & \text{if } \rho_k \leq \eta_1,\\
      \max\{\Delta_{\min}, \Delta_k\}, & \text{if } \eta_1 < \rho_k \leq \eta_2,\\
      \max\{\Delta_{\min}, \beta_2 \Delta_k\}, & \text{if } \rho_k > \eta_2.
     \end{cases}
    \end{align*}
    Set $k = k+1$.
   \ENDIF
  \ENDFOR
 \end{algorithmic}
\end{algorithm}

\begin{remark}
 Completely analogously to Lemma~\ref{lem:reform}, one shows that the minimum on the right hand side of \eqref{eq:statmeas}
 equals $-\min_{\|d\|\leq 1} \phi(u_k, \Delta_k;d)$, which is the stationarity measure from Assumption~\ref{assu:model}(\ref{it:cstat}).
\end{remark}

\begin{theorem}
 Assume that Algorithm~\ref{alg:trvi} does not terminate in finitely many steps
 and that Assumption~\ref{assu:hesse} is satisfied for all $k\in \N$.
 Then every accumulation point of the sequence of iterates is C-stationary.
\end{theorem}

\begin{proof}
 As seen in Lemma~\ref{lem:reform}, since the inexact, but feasible solution $(d_k, \zeta_k)$ of \eqref{eq:qkvimod} satisfies
 \eqref{eq:cauchyvi}, $d_k$ also fulfills the modified Cauchy-decrease condition in \eqref{eq:tildecauchy}.
 Therefore, we can apply the results for our general Algorithm~\ref{alg:TR}.
 As Lemmata~\ref{lem:Srabl}, \ref{lem:obermenge}, and \ref{lem:Gapprox} show, the conditions in Assumptions~\ref{assu:approxsubdiff}
 and \ref{assu:derivS}, that guarantee the convergence results for our trust-region algorithm applied to problems with composite functions
 as in \eqref{eq:optctrl}, are fulfilled in this concrete setting. Thus, Corollary~\ref{cor:convcomp} yields the claim.
\hfill\end{proof}

%%%%%%%%%%%%%%%%%%%%%%%%%%%%%%%%%%%%%%%%%%%%%%%%%%%%%%%%%%%%%

\section{Numerical results} \label{section: experiments}
In this section we verify the convergence properties of the proposed trust-region algorithm by means of two different examples. The first one is a toy problem in $\mathbb R^2$, for which the solution can be explicitely obtained and, moreover, the convergence hypotheses of the method analytically proved. Our purpose for this first experiment is to verify how the algorithm behaves in nondifferentiable points, i.e., biactive points, either when an optimal solution is biactive or when the algorithm has to move from such a point to further decrease the cost function.

The second example is concerned with the optimization of a variational inequality of the second kind arising from the discretization of an optimal control problem. The nondifferentiability in the variational inequality consists of the discretized $L^1$ norm of the state variable. The design variable (control) is the right hand side of the inequality, and represents a distributed control force on the whole geometric domain (see \cite{delosreyesmeyer} for further details).

Unless otherwise specified, the used trust-region parameters are: $\eta_1=0.25$, $\eta_2=0.75, \beta_1=0.5, \beta_2=1.1$  and the initial radius for the algorithm was set to $\Delta_0=1$. We use a positive definite BFGS second order matrix $H_k$, which is updated in every successful trust-region step. For the fraction of Cauchy decrease, we consider the parameter $\mu=0.8$. The algorithm stops whenever $\frac{|u_{k+1}-u_k|}{|u_0|}$ and the trust-region radius are smaller than a given tolerance. To accelerate the method we also compute the quasi Newton-step $-H_k^{-1}g$ and apply a dogleg strategy (see, e.g., \cite{kelley}).

\subsection*{Experiment 1}
We consider here the numerical solution of a toy example in $\mathbb R^2$ to illustrate the main problem and algorithm features. Let us consider the simplified variational inequality:
\begin{equation}
2y(v-y)+|v|-|y| \geq u(v-y), \quad \forall v \in \mathbb R,
\end{equation}
whose solution can be obtained using the soft thresholding operator and is given in closed form by:
\begin{equation}
y=
\begin{cases}
1/2(u-1) & \text{if }u \geq 1,\\
0 & \text{if }u \in [-1,1],\\
1/2(u+1) & \text{if }u \leq -1.
\end{cases}
\end{equation}
The solution mapping is clearly globally Lipschitz continuous and directionally differentiable. The directional derivative at $u$ in direction $h$ is given by $\eta \in \KK(y)$ solution of
\begin{equation}
2 \eta (v-\eta) \geq h(v-\eta), \quad \forall v \in \KK(y),
\end{equation}
where $\KK(y)$ is the convex cone defined by
\begin{equation}
\KK(y):=\{ v \in \mathbb R: v=0 \text{ if }|q|<1; vq \geq 0 \text{ if }y=0 \text{ and }|q|=1 \}.
\end{equation}
Since for this simplified case the biactive set corresponds to the cases $u \in \{-1,1\}$ and the set where $|q| <1$ is the same as $u \in (-1,1)$, the cone may also be written as
\begin{align}
\KK(y)&=\{ v \in \mathbb R: v=0 \text{ if }u \in (-1,1); v \geq 0 \text{ if }u=1, v \leq 0 \text{ if }u=-1 \}\\
&= \left\{ v \in \mathbb R: v \begin{cases}
\geq 0 & \text{if }u = 1,\\
= 0 & \text{if }u \in (-1,1),\\
\leq 0 & \text{if }u = -1.
\end{cases} \right\}
\end{align}
From the projection formula on convex sets it then follows that
\begin{equation}\label{eq: projection formula simplified VI}
\eta=\mathcal P_{\KK(y)} (\eta-c(2 \eta -h))=\mathcal P_{\KK(y)} ((1-2c) \eta +ch)), \quad \forall c>0.
\end{equation}

Considering the quadratic cost function
\begin{equation}
J(y,u)= \frac{1}{2} |y-z_d|^2+ \frac{\alpha}{2} |u|^2,
\end{equation}
we may define the reduced objective $f(u):=J(y(u),u)$, which, thanks to the Lipschitz continuity of the solution mapping, is a locally Lipschitz function. The directional derivative is then given by
\begin{equation}
f'(u)h = (y-z_d)^T \eta+ \alpha u^T h
\end{equation}
% In this case, the verification of the trust--region convergence hypotheses can be done directly, mainly for \eqref{eq:restglied}.
Taking the particular value $c=1/2$ in the projection formula \eqref{eq: projection formula simplified VI} we get that
\begin{equation}
\eta=  \mathcal P_{\KK(y)} \left( \frac{h}{2} \right)=
\begin{cases}
\frac{1}{2} \max(0,h) & \text{if }u = 1,\\
0 & \text{if }u \in (-1,1),\\
\frac{1}{2} \min(0,h) & \text{if }u = -1,\\
\frac{1}{2}h & \text{elsewhere.}
\end{cases}
\end{equation}
Consequently,
\begin{equation} \label{eq:directional derivative toy example}
f'(u)h=
\begin{cases}
\frac{1}{2} (y-z_d) \max(0,h)+ \alpha h & \text{if }u = 1,\\
\alpha u h & \text{if }u \in (-1,1),\\
\frac{1}{2} (y-z_d) \min(0,h) - \alpha h & \text{if }u = -1,\\
\frac{1}{2}(y-z_d)h+\alpha u h & \text{elsewhere.}
\end{cases}
\end{equation}

For the solution of this particular instance, we consider the Algorithm \ref{alg:trvi} with the direction $g \in \partial_Bf(u)$ given by $g=-(p+\alpha u)$, where the adjoint state $p$ is computed through
\begin{equation}
p_i=\begin{cases}
0 &\text{ if }i \not \in \II\\
(y-z_d)_i &\text{ if }i \in \II.
\end{cases}
\end{equation}
Since in this case we are working with a single real variable, the direction can be explicitely written as
\begin{equation}
g= \begin{cases}
-\alpha u  & \text{if }u \in [-1,1],\\
-\frac{1}{2}(y-z_d)-\alpha u & \text{elsewhere.}
\end{cases}
\end{equation}

Using \eqref{eq:directional derivative toy example} we may compute the directional derivative along the Bouligand direction $g$ yielding
\begin{align}
f'(u)g & =
\begin{cases}
\frac{1}{2} (y-z_d) \max(0,-\alpha u)- \alpha^2 u^2 & \text{if }u = 1,\\
-\alpha^2 u^2 & \text{if }u \in (-1,1),\\
\frac{1}{2} (y-z_d) \min(0,-\alpha u) - \alpha^2 u^2 & \text{if }u = -1,\\
- \left[ \frac{1}{2}(y-z_d)+\alpha u \right]^2 & \text{elsewhere.}
\end{cases}\\
& =
\begin{cases}
-\alpha^2 u^2 & \text{if }u \in [-1,1],\\
- \left[ \frac{1}{2}(y-z_d)+\alpha u \right]^2 & \text{elsewhere.}
\end{cases}
\end{align}
Consequently, $f'(u)g = -|g|^2$ for the direction considered, i.e., the element in Assumption~\ref{assu:derivS} is explicitely given.

If the trust-region radius becomes smaller than $\Delta_{\min}=1e-2$, problem \eqref{eq:qkvimod} is solved for all possible variations of the possibly biactive set according to Definition~\ref{def:possbiact}. For the current toy problem the latter reduces to solve two auxiliary QP problems, and one extra for determining $\psi(u_k,\Delta_k)$. The auxiliary problems are solved using a Sequential Least Squares Programming (SLSQP) algorithm available in SciPy's Optimization library.

% TODO: stopping criteria

Since in this case the solution can be computed analytically, the different properties of the algorithm can also be easily verified. This involves in particular the attainability of minima, mainly when these points are biactive, or the escape from biactive points when they are reached along the iterative process.

In Figure \ref{figire: toy example functions} the solution operator and the composite cost function are plotted. As it can be oberved, the resulting objective function is piecewise differentiable with two local minima at $\bar u_1=-1$ and $\bar u_2=(4 \alpha u_d+2 z_d+1)/(4 \alpha +1)$. The points $u=-1$ and $u=1$ are biactive, and only one of them corresponds to a local minimum for the problem.
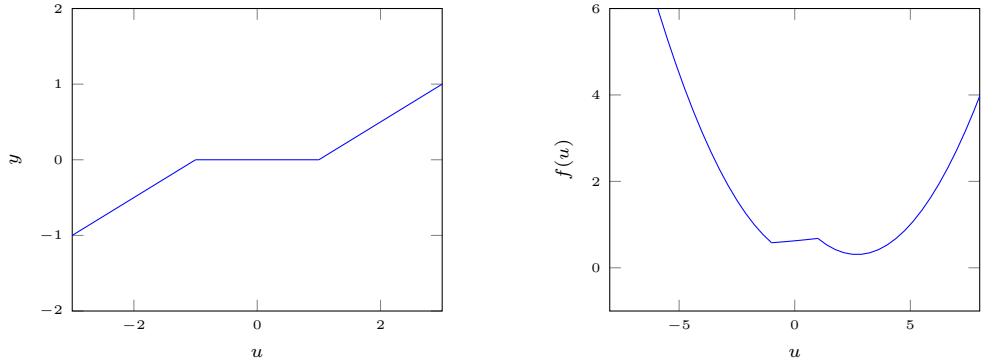
\begin{figure}[H]
  \begin{tikzpicture}
  	\begin{axis}[
      xmin=-3,   xmax=3,
ymin=-2,   ymax=2,
xlabel={$u$},
ylabel={$y$},
ylabel style = {font=\scriptsize},
xlabel style = {font=\scriptsize},
yticklabel style = {font=\tiny},
xticklabel style = {font=\tiny}]
  	]
  	% use TeX as calculator:
  	\addplot[blue,domain=-3:-1,mark=] {0.5*(x +1)};
    \addplot[blue,domain=-1:1,mark=] {0*x};
    \addplot[blue,domain=1:3,mark=] {0.5*(x -1)};
  	\end{axis}
  \end{tikzpicture}
  \hfill
  \begin{tikzpicture}
  	\begin{axis}[
      xmin=-8,   xmax=8,
ymin=-1,   ymax=6,
xlabel={$u$},
ylabel={$f(u)$},
ylabel style = {font=\scriptsize},
xlabel style = {font=\scriptsize},
yticklabel style = {font=\tiny},
xticklabel style = {font=\tiny}]
  	]
  	% use TeX as calculator:
  	\addplot[blue,domain=-10:-1,mark=] {0.5*(0.5*(x+1)-1)^2+0.01/2*(x+5)^2};
    \addplot[blue,domain=-1:1,mark=] {0.5+0.01/2*(x+5)^2};
    \addplot[blue,domain=1:10,mark=] {0.5*(0.5*(x-1)-1)^2+0.01/2*(x+5)^2};
  	\end{axis}
  \end{tikzpicture}
  \caption{Solution operator (left) and cost function (right). Parameters: $\alpha =0,01$, $z_d=1$, $u_d=-5$.} \label{figire: toy example functions}
\end{figure}

If the initial point of Algorithm~\ref{alg:trvi} is chosen below or equal than one, then the method converges towards the local minimum $\bar u_1$, while it converges to $\bar u_2$ if the starting iterate is chosen greater than one. In Figure \ref{fig: iterations toy} the number of trust-region iterations for different initial points and $\alpha$ values is depicted. It can be easily observed that reaching the nonsmooth minimum requires significantly more iterations than reaching the smooth one, which is intuitively expected. Despite of that, the total number of trust-region iterations stays below 20 in all cases.

\begin{figure}
\centering
\includegraphics[height=7.5cm,width=10cm]{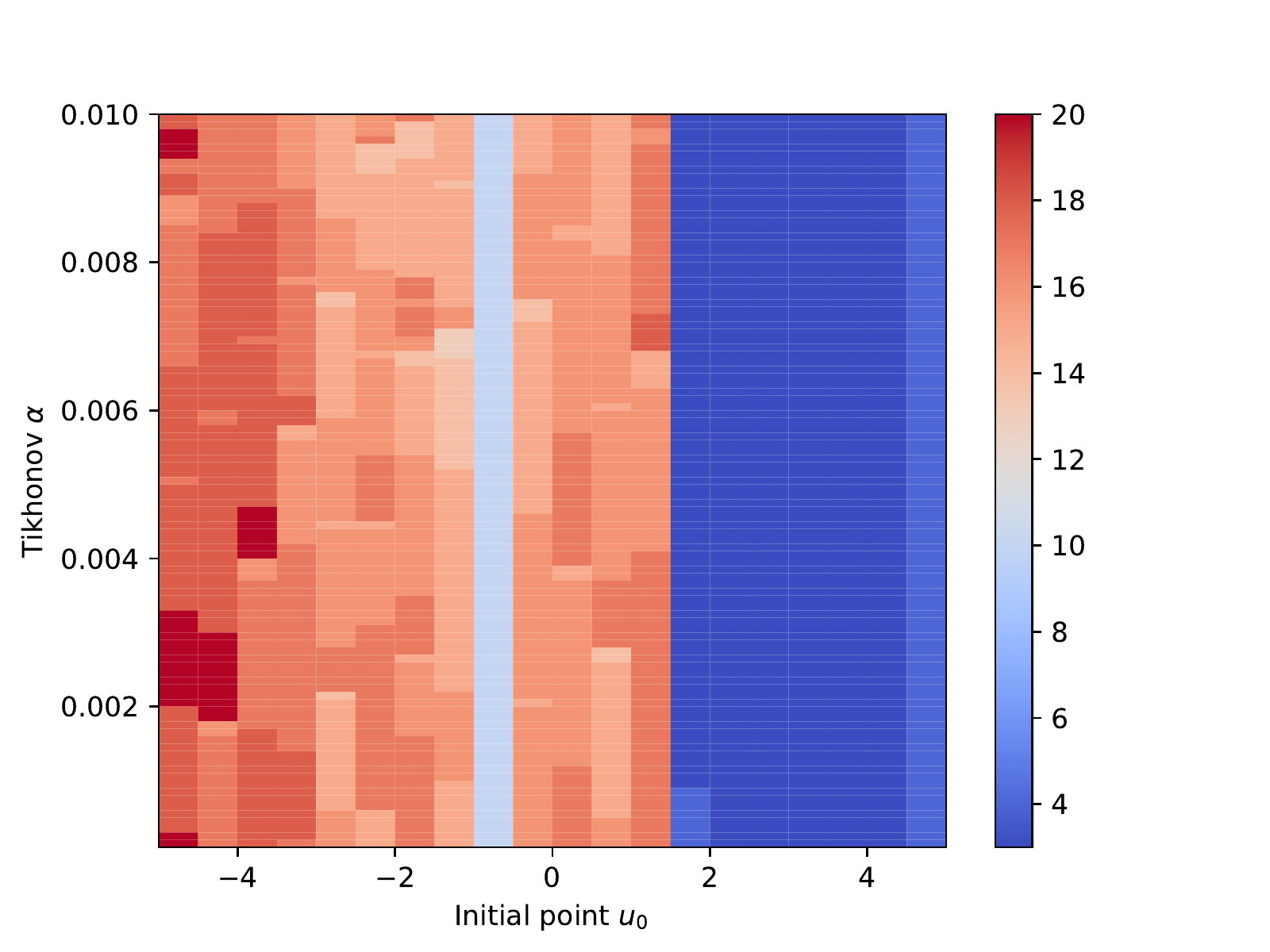}
\caption{Number of trust-region iterations, depending on the initial guess $u_0 \in [-5,5]$ and the Tikhonov parameter $\alpha \in [1e-4,1e-2]$.} \label{fig: iterations toy}
\end{figure}

\subsection*{Experiment 2}
Let us start this experiment by recalling problem \eqref{eq:piv}:
\begin{equation}
 \left.
 \begin{aligned}
  \min_{y,u\in \R^n} & \quad J(y,u)\\
  \text{s.t.} & \quad \dual{\nu^{-1}A y}{v-y} + |v|_1 - |y|_1 \geq \dual{\nu^{-1} u}{v-y}, \quad \forall\, v\in \R^n,
 \end{aligned}
 \qquad\right\}
\end{equation}
where $|.|_1$ denotes the 1-norm, i.e., $|v|_1 = \sum_{i=1}^n |v_i|$. We consider the matrix $A$ as the homogeneous finite differences discretization matrix of the negative two-dimensional Laplace operator with homogeneous Dirichlet boundary conditions on the domain $\Omega=(0,1)^2$. The 1-norm, together with the weight $\nu$, is responsible for the level of sparsity in the computed state. When $\nu$ increases, the state becomes sparser up to a certain threshold where the state is identically zero.

As cost function we choose the classical tracking--type objective
\begin{equation*}
  J(y,u): = \frac{1}{2} \|y- z_d\|^2+ \frac{\alpha}{2} \|u\|^2,
\end{equation*}
where $\alpha>0$ is the Tikhonov weight and $z_d=\chi_{x_1 \geq 1}$. In Figure \ref{figure: states for different thresholds} the controlled states for different values of $\nu$ are shown. As expected, it can be observed that as $\nu$ becomes larger, the region where the state has zero value also increases. This is something expected from the structure of the variational inequality constraint. Concerning the behaviour of the biactive sets, in Figure \ref{figure: biactive sets for different thresholds} these sets are plotted for different $\nu$ values. It can be realized that the biactive region is not dismissible in practice and should be carefully handled.
\begin{figure}[h]
  \begin{center}
\includegraphics[height=4cm,width=6cm]{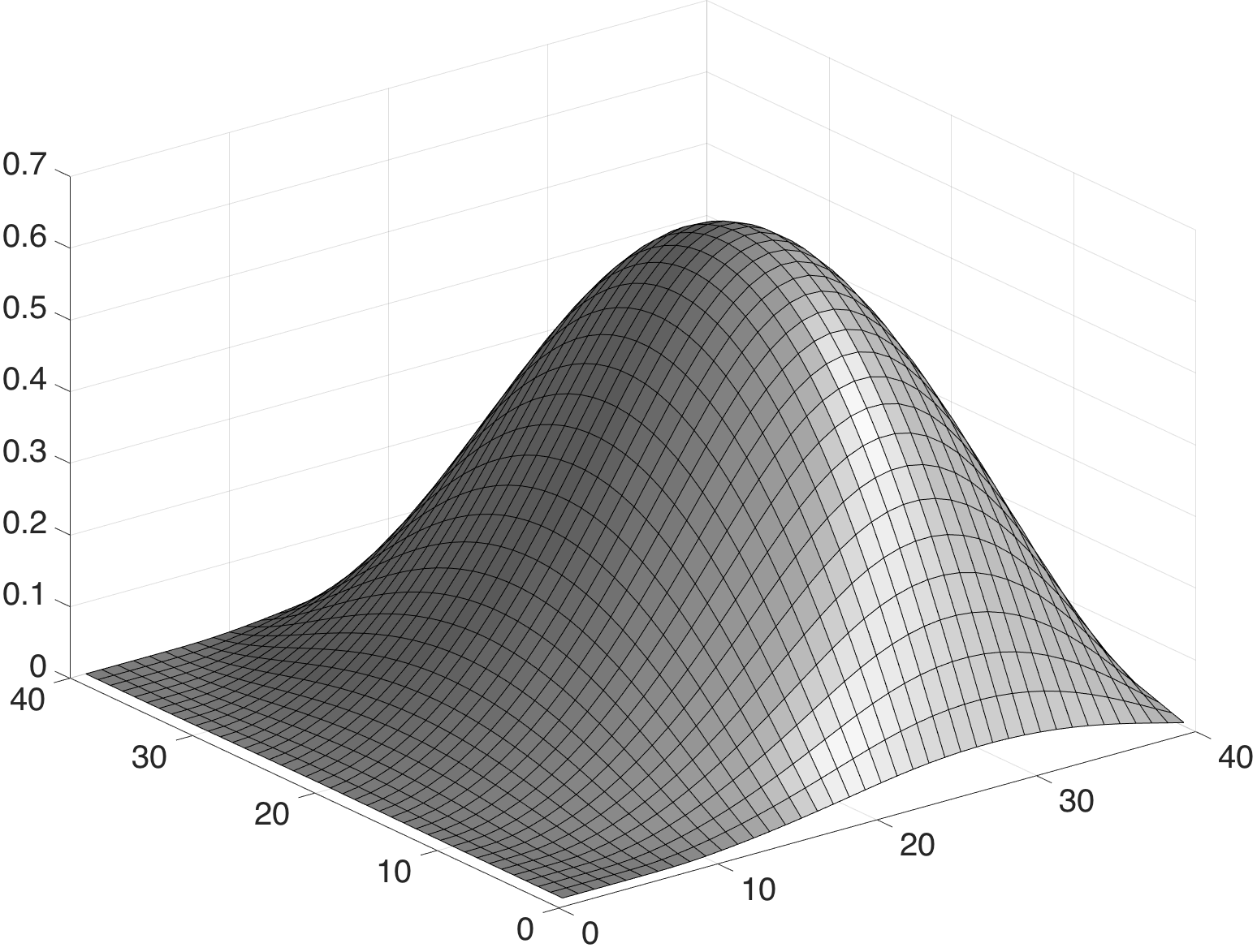} \hfill
\includegraphics[height=4cm,width=6cm]{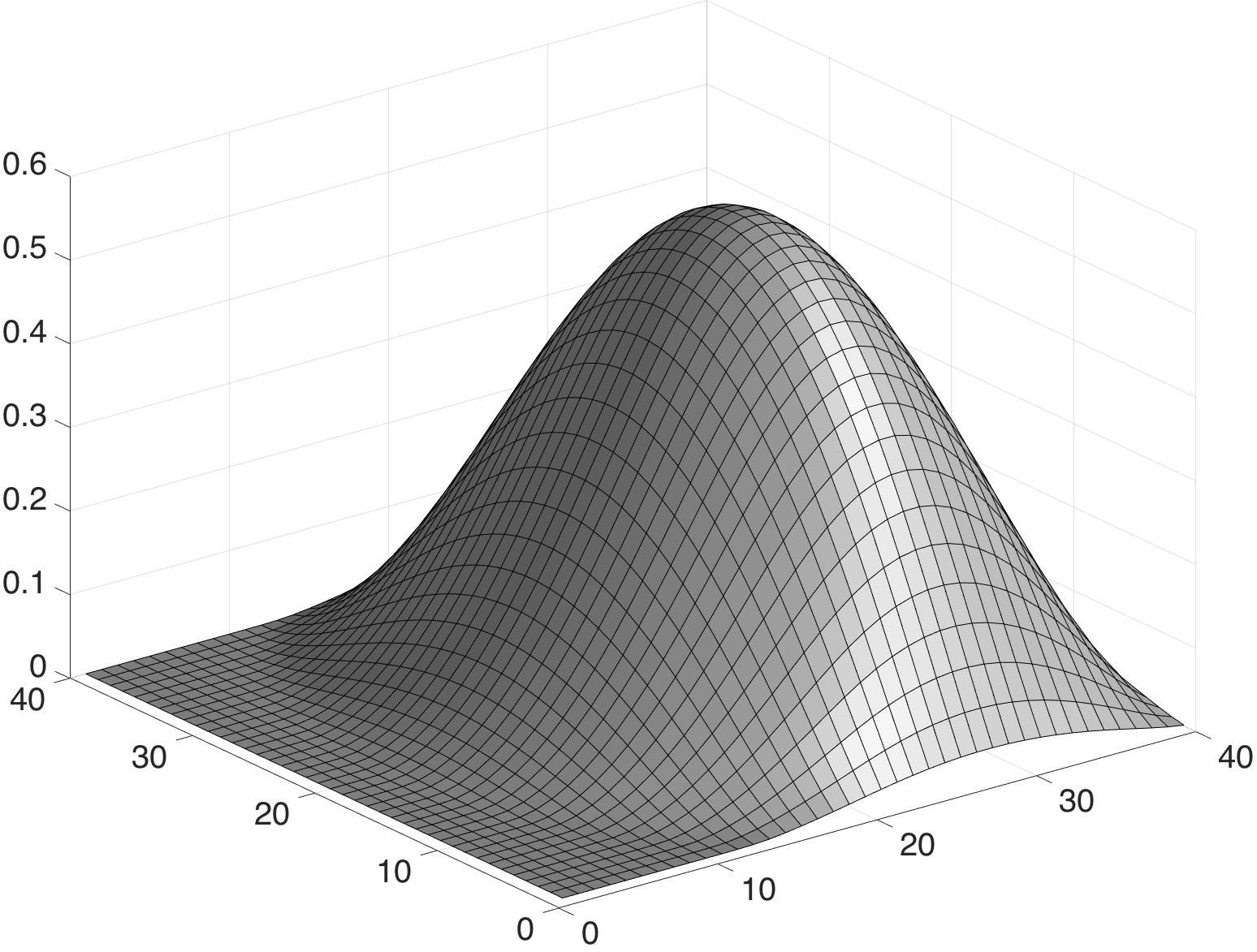}\\
\includegraphics[height=4cm,width=6cm]{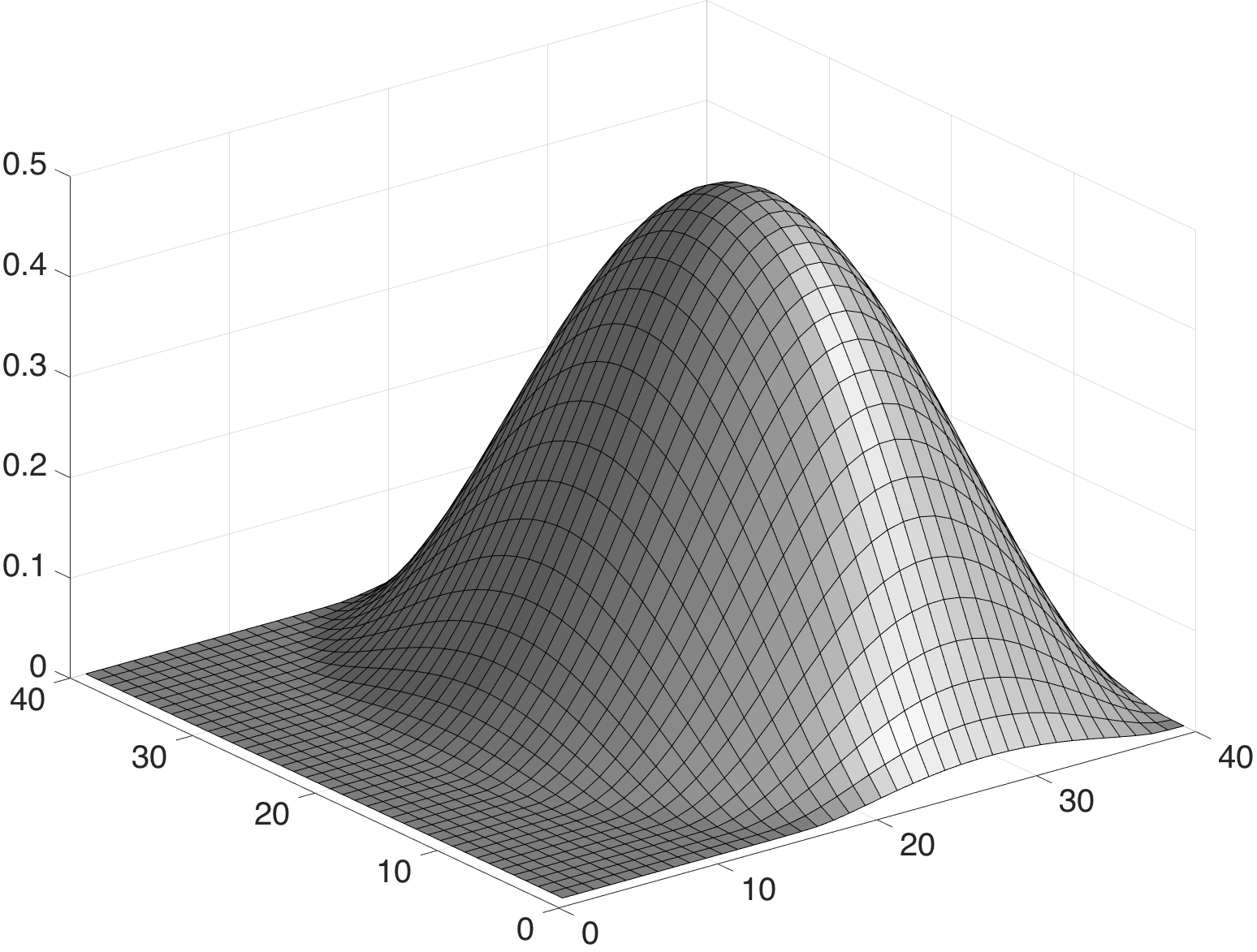} \hfill
\includegraphics[height=4cm,width=6cm]{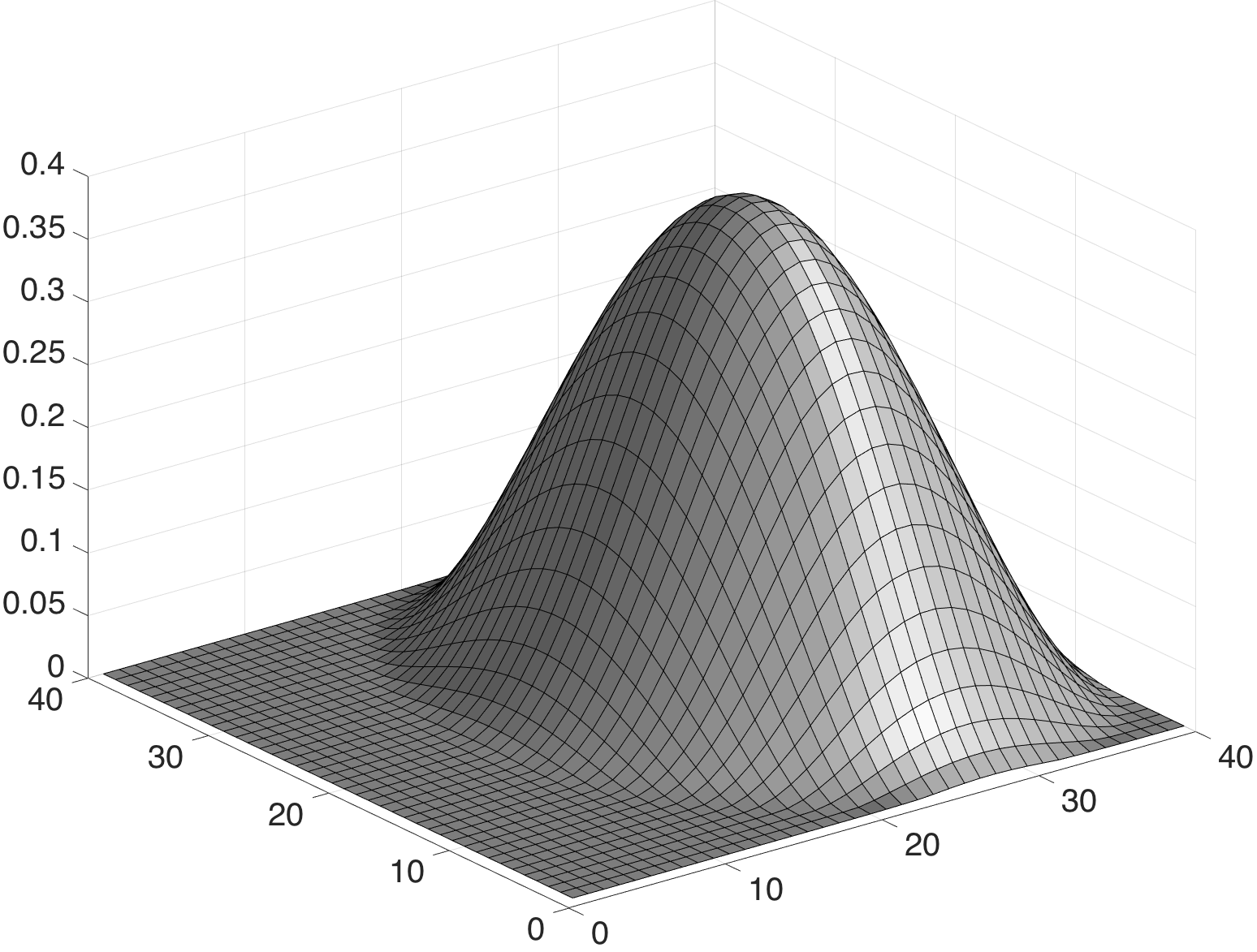}\\
\caption{Optimal controlled state $y$ for different parameters $\nu$. From the left upper corner to the lower right corner: $\nu=4$, $\nu=8$, $\nu=12$, $\nu=18$. Tikhonov parameter: $\alpha = 0.001$; mesh size step $h=1/41$.}
\end{center} \label{figure: states for different thresholds}
\end{figure}

\begin{figure}
  \begin{center}
\includegraphics[height=3.5cm,width=4cm]{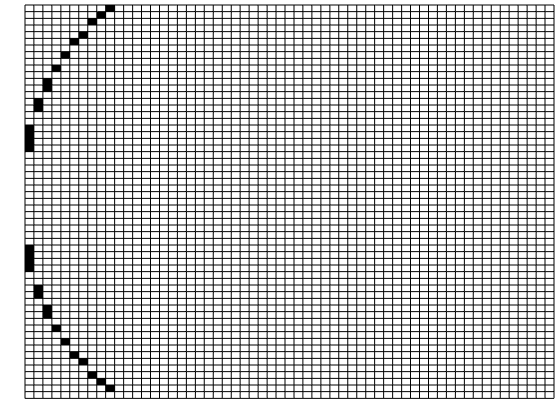} \hspace{0.5cm}
\includegraphics[height=3.5cm,width=4cm]{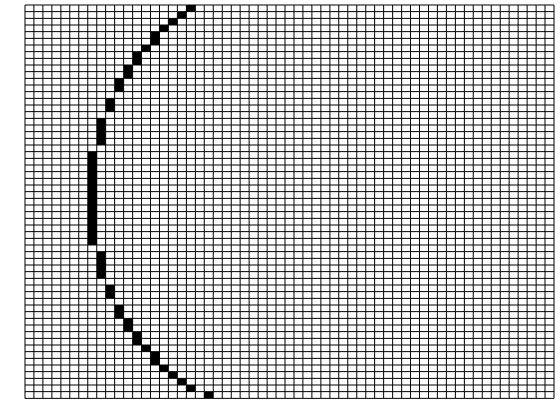}\\
\includegraphics[height=3.5cm,width=4cm]{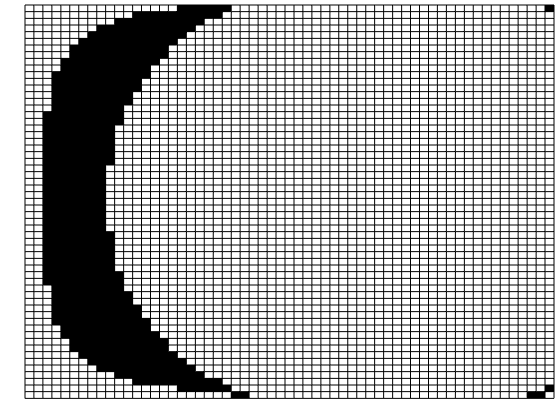} \hspace{0.5cm}
\includegraphics[height=3.5cm,width=4cm]{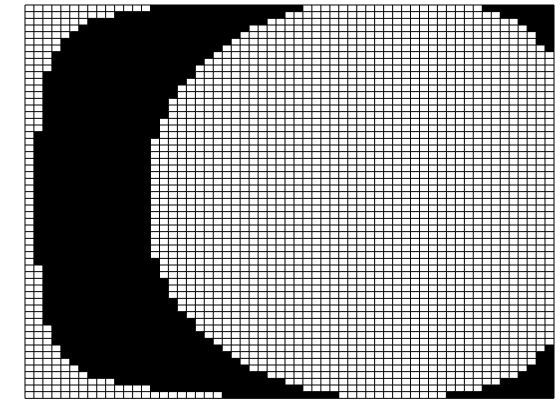}
\caption{Biactive sets for different parameters $\nu$. From the left upper corner to the lower right corner: $\nu=4$, $\nu=8$, $\nu=12$, $\nu=18$. Tikhonov parameter: $\alpha = 1\mathrm{e}-3$; mesh size step $h=1/61$.}
\end{center} \label{figure: biactive sets for different thresholds}
\end{figure}

\emph{Implementation details.}
The computational domain $\Omega$ was discretized using homogeneous finite differences with step size $h=\frac{1}{n+1}$, where $n$ is the number of inner discretization points. The resulting stiffness matrix $A$ is therefore penta-diagonal, symmetric and positive definite. For the solution of the linear systems we therefore considered MATLAB's exact sparse solver. The trust-region radius lower bound was set to $\Delta_{\min}=1e-3$

The main costly step in our algorithm is the solution of the variational inequality constraint. To improve the total computing time we therefore consider the inexact solution of the variational inequality constraint in the first trust-region iterations, and the exact solution of the problem in the last ones. Specifically, up to a tolerance of $1\mathrm{e}-2$ for the norm of the trust-region residuum, the variational inequality is solved by means of a semismooth Newton method (with a Huber regularization \cite{huber2011robust}). After that, and until the required tolerance is reached (typically $1\mathrm{e}-6$), the lower level problem is solved using a recently proposed orthantwise method \cite{dlrlm}. The orthantwise method stops whenever the norm of the pseudo-gradient is smaller than $1\mathrm{e}-7$.

\emph{Performance.}
The proposed inexact trust--region algorithm does not have significant difficulties for solving the optimization problem at hand, even when the solution has large biactive sets. The total iteration number of the trust--region method is provided in Table \ref{table: TR iterations for meshes} for two medium size meshes and different values of the Tikhonov parameter $\alpha$ and the coefficient $\nu$. The number of trust-region iterations where the semismooth Newton solver was used is registered in parenthesis.
% In Figure \ref{fig:mesh} also a plot of the number of iterations for three meshes and different values of $\nu$ is provided. Both from the table and the figure, it can be observed that the number of iterations does not significantly vary for the different meshes considered.

% \begin{center}
% \begin{table}
% \scalebox{0.58}{
% \begin{tabular}{|l|c|c|c|c|c|c|c|c|c|c|c|c|} \hline
%   & \multicolumn{4}{c|}{$h=1/20$} &\multicolumn{4}{c|}{$h=1/40$} &\multicolumn{4}{c|}{$h=1/60$}\\ \hline
%   $g$ &4 &8 &12 &18 &4 &8 &12 &18 &4 &8 &12 &18\\    \hline
%   $1e-1$& 60(40) &60(40) &60(40) &60(40) &68(48) &71(48) &68(48) &68(48) &72(52) &75(52) &75(52) &75(52)\\    \hline
%   $1e-2$& 73(53) &73(39) &74(39) &74(39) &81(61) &82(47) &82(47) &82(47) &89(66) &86(51) &87(51) &86(51)\\    \hline
%   $1e-3$& 82(50) &71(49) &67(47) &66(48) &90(63) &92(65) &87(73) &77(57) &92(70) &140(170) &120(104) &83(63)\\    \hline
%   $1e-4$& 83(68) &79(65) &120(106) &112(99) &201(78) &-- &129(88) &116(90) &-- &{\color{red}232(80)} &203(104) &--\\    \hline
%   $1e-5$& 147(133) &153(116) &288(245) &244(174) &214(199) &249(235) &215(170) &233(114) &-- &--) &200(96) &--\\    \hline
% \end{tabular}}
% \end{table}
% \end{center}

\begin{center}
\begin{table}
\scalebox{0.9}{
\begin{tabular}{|l|c|c|c|c||c|c|c|c|} \hline
  & \multicolumn{4}{c||}{$h=1/20$} &\multicolumn{4}{c|}{$h=1/40$}\\ \hline
  \backslashbox{$\alpha$}{$\nu$} &4 &8 &12 &18 &4 &8 &12 &18\\    \hline
  $1e-1$& 46(20) &46(20) &46(20) &46(20) &53(25) &53(25) &53(25) &53(25)\\    \hline
  $1e-2$& 26(22) &69(20) &69(20) &69(20) &30(26) &76(24) &76(24) &76(24)\\    \hline
  $1e-3$& 35(31) &28(23) &26(21) &72(08) &46(41) &36(25) &29(25) &79(08)\\    \hline
  $1e-4$& 35(30) &38(33) &41(35) &72(08) &104(40) &105(37) &85(42) &79(08)\\    \hline
\end{tabular}}
\caption{Number of iterations in two different meshes for different values of $\alpha$ and $\nu$.} \label{table: TR iterations for meshes}
\end{table}
\end{center}

% \begin{figure}
% \begin{center}
%   \begin{tikzpicture}
%   	\begin{axis}[
%   		xlabel=$\nu$,
%   		ylabel=Iterations]
%   	\addplot[color=red,mark=x] coordinates {
%   		(4,37)
%   		(8,29)
%   		(12,28)
%       (15,69)
%   		(18,76) %zero
%     	};
%       \addlegendentry{n=30}
%       \addplot[color=blue,mark=*] coordinates {
%     		(4,52)
%     		(8,86)
%     		(12,32)
%         (15,46)
%     		(18,83)
%       	};
%         \addlegendentry{n=60}
%         \addplot[color=gray,mark=square] coordinates {
%       		(4,58)
%       		(8,88)
%       		(12,35)
%           (15,58)
%       		(18,88)
%         	};
%           \addlegendentry{n=90}
%   	\end{axis}
%   \end{tikzpicture}
%   \caption{Number of trust--region iterations for different mesh size steps. Tikhonov parameter: $\alpha=0.001$.}
%   \label{fig:mesh}
% \end{center}
% \end{figure}

\section{Conclusions}
We have presented a non-smooth trust--region algorithm for solving optimization problems with locally Lipschitz continuous functions and,
under suitable assumptions, we prove convergence of the iterates to a C-stationary point.
The structure of the trust-region subproblem is dependent on the size of the current trust-region radius and allows to escape from non-differentiable points which are not necessarily local minima.

As a particular instance we have considered optimization problems with a class of variational inequality constraints and proposed a computable model function to be used along the trust--region iterations.
The construction of this model is based on a precise characterization of the Bouligand-subdifferential associated with the
solution operator of the variational inequality. Moreover, thanks to the structure of these types of problems, the computation of the solution to the trust-region subproblem \eqref{eq:qkvimod} can be carried out just by adding a finite number of inequality constraints dependent on the size of the biactive set.

The proposed algorithm is general enough to deal with several classes of problems. We tested it for the particular case of optimization problems constrained by variational inequalities of the second kind and its performance was succesfully verified. The investigation of the behaviour of the algorithm for other types of problem is a matter of future research.

%%%%%%%%%%%%%%%%%%%%%%%%%%%%%%%%%%%%%%%%%%%%%%%%%
\begin{appendix}

\section{Proof of Lemma~\ref{lem:counterex1d}}\label{sec:counterex1d}

First we note that the setting in \eqref{eq:params} and \eqref{eq:delta0} fulfills the requirements
of the initialization in Step~\ref{it:init} of Algorithm~\ref{alg:TR}. In particular,
the condition $\beta_1 + \beta_1 \beta_2 < 1$ ensures that the interval for $x_0$ is non-empty and
that $\Delta_0 \in (0,1]$. It moreover guarantees that
\begin{equation*}
 \theta := \Big(\frac{b}{a} - 1\Big)\frac{\beta_1}{\beta_1\beta_2-1} + \frac{b}{a}
 %= \Big(\frac{b}{a} - 1\Big)\frac{x_0}{\Delta_0} + \frac{b}{a}
 \in (0,1)
\end{equation*}
so that the interval $[\theta, 1)$ is non-empty and $\eta_1$ can be chosen as required in \eqref{eq:params}.

Throughout the proof, we assume that the trust-region subproblems in \eqref{eq:qk} and \eqref{eq:tildeqk}
are solved exactly (which is not restrictive at all, as they represent one-dimensional linear programs).
We show by induction that, for all $k\in \N_0$, it holds
\begin{alignat*}{9}
 & \text{(i)} & \;\;  x_{2k} &= (\beta_1\beta_2)^k x_0, & \quad
 & \text{(ii)} & \;\; x_{2k+1} & = x_{2k}, &\quad
 & \text{(iii)} & \;\;  \Delta_{2k} &= (\beta_1\beta_2)^k \Delta_0, \\
 & \text{(iv)} & \;\;  \Delta_{2k+1} &= \beta_1(\beta_1\beta_2)^k \Delta_0, & \quad
 & \text{(v)} & \;\;  \rho_{2k} & \equiv \theta, & \quad
 & \text{(vi)} & \;\;  \rho_{2k+1} &\equiv 1.
\end{alignat*}
Let $k = 0$. Then (i) and (iii) are obviously true. Furthermore, since $\Delta_0 < \Delta_{\min}$
by \eqref{eq:delta0}, the first trust-region subproblem is given by \eqref{eq:tildeqk}, whose
exact solution is $d_0 = \Delta_0$ because of $\widetilde\phi(x_0, \Delta_0; d) = - a d$ (as $x_0 < 0$)
and $H_0 = 0$. Therefore, thanks to \eqref{eq:delta0}, we arrive at
\begin{equation}\label{eq:x0delta0}
 x_0 + d_0 = x_0 + \Delta_0 = x_0 \Big(1 + \frac{\beta_1\beta_2 - 1}{\beta_1}\Big) \in (0, 1)
\end{equation}
Consequently, on account of \eqref{eq:delta0} and $\psi(x_0, \Delta_0) = a > a \Delta_0$, \eqref{eq:rhomod} gives
\begin{equation*}
 \rho_0 = \frac{f(x_0) - f(x_0 + d_0)}{- \widetilde\phi(x_0, \Delta_0; d_0)}
 =\frac{-a x_0 + b(x_0 + \Delta_0)}{a \Delta_0}  = \theta \leq \eta_1
\end{equation*}
so that (v) is true for $k=0$. The update in \eqref{eq:upx} and \eqref{eq:updelta} thus implies
$x_1 = x_0$ and $\Delta_1 = \beta_1\Delta_0$, i.e., (ii) and (iv) for $k=0$.
To compute $\rho_1$, observe that again $d_1 = \Delta_0$
is the solution of the trust-region subproblem and $x_1 + \Delta_1 = \beta_1\beta_2x_0 < 0$. Hence,
\begin{align*}
 \rho_1 = \frac{f(x_1) - f(x_1 + \Delta_1)}{-\widetilde\phi(x_1,\Delta_1)}
 = \frac{-a x_1 + a(x_1 + \Delta_1)}{a \Delta_1} = 1,
\end{align*}
giving (vi) for $k=0$.

Now let $k \in \mathbb{N}_0$ be arbitrary and suppose that (i)--(vi) is true for that $k$.
Since $x_{2k + 1} = x_{2k} = (\beta_1\beta_2)^kx_0 < 0$, we obtain $d_{2k+1} = \Delta_{2k + 1}$
as the is the exact solution of the trust-region subproblem.
Thus, from $\rho_{2k+1} = 1$ and the update formulas in \eqref{eq:upx} and \eqref{eq:updelta}
it follows that $x_{2(k+1)} = x_{2k + 1} + \Delta_{2k + 1} = (\beta_1\beta_2)^kx_0 + \beta_1(\beta_1\beta_2)^k\Delta_0 = (\beta_1\beta_2)^{k+1}x_0$ and $\Delta_{2(k + 1)} = \beta_2\Delta_{2k + 1} = (\beta_1\beta_2)^{k+1}\Delta_0$, i.e.,  (i) and (iii) for $k+1$.
Thus, $x_{2(k+1)} < 0$ and $\Delta_{2(k+1)}$ is the exact solution of the trust-region subproblem.
Moreover, \eqref{eq:x0delta0} yields
$x_{2(k+1)} + \Delta_{2(k+1)} = (\beta_1\beta_2)^{k+1}(x_{0} + \Delta_{0}) \in (0,1)$ and consequently
\begin{align*}
 \rho_{2(k+1)} =& \frac{f(x_{2(k+1)}) - f(x_{2(k+1)} + d_{2(k+1)})}{-\widetilde\phi(x_{2(k+1)},d_{2(k+1)})} \\
 =& \frac{-a x_{2(k+1)} + b(x_{2(k+1)} + \Delta_{2(k+1)})}{a \Delta_{2(k+1)}} \\
 =& -\frac{(\beta_1\beta_2)^{k+1}x_0}{(\beta_1\beta_2)^{k+1}\Delta_0}
 + \frac{b}{a}\frac{(\beta_1\beta_2)^{k+1}x_0 + (\beta_1\beta_2)^{k+1}\Delta_0}{(\beta_1\beta_2)^{k+1}\Delta_0}
 = \theta \leq \eta_1,
\end{align*}
which is (v) for $k+1$. Therefore, the update formulas in \eqref{eq:upx} and \eqref{eq:updelta} give (ii) and (iv) for $k+1$.
Therefore, it follows that
\begin{align*}
 x_{2(k+1) + 1} + \Delta_{2(k+1) + 1} &= (\beta_1\beta_2)^{k+1}x_0 + \beta_1(\beta_1\beta_2)^{k+1}\Delta_0 \\
 &= (\beta_1\beta_2)^{k+1}x_0 + (\beta_1\beta_2)^{k+1}(\beta_1\beta_2-1)x_0 = (\beta_1\beta_2)^{k+2}x_0 < 0
\end{align*}
such that
\begin{align*}
 \rho_{2(k+1) + 1}
 &= \frac{f(x_{2(k+1) + 1}) - f(x_{2(k+1) + 1} + \Delta_{2(k+1) + 1})}{-\widetilde\phi(x_{2(k+1) + 1},\Delta_{2(k+1) + 1})} \\
 &= \frac{-a x_{2(k+1) + 1} + a(x_{2(k+1) + 1} + \Delta_{2(k+1) + 1})}{a \Delta_{2(k+1) + 1}} = 1,
\end{align*}
i.e., (vi) for $k+1$, which completes the proof of (i)--(vi).
Therefore, if the local model in \eqref{eq:wrongmod} is used and the trust-region subproblems are solved exactly, then,
with the setting in \eqref{eq:params} and \eqref{eq:delta0}, we obtain $x_k \to 0$, as claimed.

By contrast, if we apply the non-local model function in \eqref{eq:correctmod}, then $x_k \to 1$, which is seen as follows:
Firstly, one shows analogously to the proofs of Lemma~\ref{lem:remainder} and \ref{lem:statmeas} that the model function in
\eqref{eq:correctmod} satisfies the conditions in Assumption~\ref{assu:model}(\ref{it:model}).
Therefore, the convergence analysis from Section~\ref{subsec:conv} applies. On the other hand, by construction of the algorithm,
the sequence of iterates stays in the level set $\{x\in \R: f(x) \leq f(x_0)\}$, which is clearly compact
for every initial point $x_0$ in case of our one-dimensional objective in \eqref{eq:objbsp}.
Therefore, there is a converging subsequence, whose limit it C-stationary by Theorem~\ref{thm:TRconv}. The only C-stationary point
of this objective however is $x = 1$ so that the uniqueness of the limit gives the convergence of the whole sequence.

\begin{remark}
 Based on \textup{(i)--(vi)}, we can now specify the failure of the local model in \eqref{eq:wrongmod}.
 Evaluating the expression in \eqref{eq:remainder} for the sequence of iterates yields
 \begin{align*}
  \frac{f(x_{2k} + d_{2k}) - f(x_{2k}) - \widetilde\phi(x_{2k}, \Delta_{2k}; d_{2k})}{\Delta_{2k}}
  =& \frac{-b (x_{2k} + \Delta_{2k}) + a x_{2k} + a \Delta_{2k}}{\Delta_{2k}} \\
  =& \frac{(a-b) (x_{0} + \Delta_{0})}{ \Delta_0} \equiv \text{const.} > 0
 \end{align*}
 although $x_{2k} \to 0$, $\Delta_{2k} \to 0$, and $\psi(x_{2k}, \Delta_{2k}) \to a > 0$, so that
 Assumption~\ref{assu:model}(\ref{it:remainder}) is not fulfilled in case of the local model.
 Since this assumption can be interpreted as a non-smooth remainder term property,
 this again demonstrates the lack of neighborhood information in case of local models.
\end{remark}

\end{appendix}

\section*{Acknowledgement}
The authors thank Stephan Walther (TU Dortmund) for elaborating the proof of Lemma~\ref{lem:counterex1d}.

%%%%%%%%%%%%%%%%%%%%%%%%%%%%%%%%%%%%%%%%%%%%%%%%%%%%%

\end{document}